\documentclass{amsart}
\usepackage[margin=1.78cm]{geometry}
\usepackage[T1]{fontenc}
\usepackage{lmodern}
\usepackage[all]{xy}
\usepackage{graphicx}
\usepackage{indentfirst}
\usepackage{bm}
\usepackage{amsthm}
\usepackage{mathrsfs}
\usepackage{latexsym}
\usepackage{amsmath}
\usepackage{amssymb}
\usepackage{color}
\usepackage{hyperref}
\usepackage{microtype}
\usepackage{tikz}
\usepackage{enumerate}
\usepackage{comment}
\usepackage{braket}
\usepackage{manfnt}
\usepackage{hyperref}
\usepackage{slashbox}
\usepackage{cite}
\allowdisplaybreaks

\usepackage{color}
\usetikzlibrary{calc}
\usetikzlibrary{shapes.geometric}
\usetikzlibrary{arrows}
\usetikzlibrary{decorations.pathmorphing}
\usetikzlibrary{decorations.text}
\usetikzlibrary{decorations.shapes}
\usetikzlibrary{decorations.fractals}
\usetikzlibrary{decorations.footprints}

\newtheorem{theorem}{Theorem}[section]

\newtheorem{proposition}[theorem]{Proposition}
\newtheorem{corollary}[theorem]{Corollary}
\newtheorem{definition}[theorem]{Definition}
\newtheorem{lemma}[theorem]{Lemma}

\newtheorem{remark}[theorem]{Remark}
\newtheorem{question}[theorem]{Question}

\newtheorem{statement}[theorem]{Statement}
\newtheorem{ttable}[theorem]{Table}
\newtheorem*{ac}{Acknowledgements}

\usepackage{lipsum}
\newcommand\blfootnote[1]{%
  \begingroup
  \renewcommand\thefootnote{}\footnote{#1}%
  \addtocounter{footnote}{-1}%
  \endgroup
}

\newcommand{\Rep}{\mathrm{Rep}}
\newcommand{\VVec}{\mathrm{Vec}}
\newcommand{\PSL}{\mathrm{PSL}}
\newcommand{\SL}{\mathrm{SL}}
\newcommand{\Max}{\mathrm{max}}
\newcommand{\FPdim}{\mathrm{FPdim}}
\newcommand{\spec}{i_0} 

\title{Interpolated family of non group-like simple integral fusion rings of Lie type}

\author{Zhengwei Liu}
\address{Z. Liu, Yau Mathematical Sciences Center, Department of Mathematics, Tsinghua University, and Yanqi Lake Beijing Institute of Mathematical Sciences and Applications, Beijing, 100084, China}
\email{liuzhengwei@mail.tsinghua.edu.cn}

\author{Sebastien Palcoux}
\address{S. Palcoux, Yanqi Lake Beijing Institute of Mathematical Sciences and Applications, Huairou District, Beijing, China}
\email{sebastien.palcoux@gmail.com}
\urladdr{https://sites.google.com/view/sebastienpalcoux}

\author{Yunxiang Ren}
\address{Y. Ren, Department of Physics, Harvard University, Cambridge, 02138, USA}
\email{renyunxiang@gmail.com}

\begin{document}

\maketitle

\begin{abstract}%
This paper is motivated by the quest of a non-group irreducible finite index depth $2$ maximal subfactor. We compute the generic fusion rules of the Grothendieck ring of $\Rep(\PSL(2,q))$, $q$ prime-power, by applying a Verlinde-like formula on the generic character table. We then prove that this family of fusion rings $(\mathcal{R}_q)$ interpolates to all integers $q \ge 2$, providing (when $q$ is not prime-power) the first example of infinite family of non group-like simple integral fusion rings. Furthermore, they pass all the known criteria of (unitary) categorification. This provides infinitely many serious candidates for solving the famous open problem of whether there exists an integral fusion category which is not weakly group-theoretical. We prove that a complex categorification (if any) of an interpolated fusion ring $\mathcal{R}_q$ (with $q$ non prime-power) cannot be braided, and so its Drinfeld center must be simple. In general, this paper proves that a non-pointed simple fusion category is non-braided if and only if its Drinfeld center is simple; and also that every simple integral fusion category is weakly group-theoretical if and only if every simple integral modular fusion category is pointed.
\end{abstract}

\blfootnote{{\bf MSC2020:} primary 46L37, 18M20; secondary 20C33, 20D06.

{\bf Keywords:} Subfactor, fusion category, finite simple group of Lie type, character table, interpolation, simple integral fusion ring.
}

\section{Introduction}

A subfactor \cite{JonesSunder} encodes a Galois-like quantum generalization of the notion of symmetry \cite{EvaKawa, Jones2009}, analogous to a field extension, where its planar algebra \cite{Jon99} and its fusion category \cite{EGNO15} are the analogous of the Galois group and its representation category, respectively. The index of a subfactor \cite{Jon83} is multiplicative with its intermediates, so a subfactor without proper intermediate (called a \emph{maximal subfactor} \cite{Bisch}) can be seen as a quantum analogous of the notion of prime number (about a quantum analogous of the notion of natural number, see \cite{Pal1,Pal2,Pal3}). A group is a notion of classical symmetry. A finite group subfactor planar algebra is in the bigger class of irreducible finite index depth $2$ subfactor planar algebras, which is exactly the class of finite quantum group subfactor planar algebras \cite{Szy94,Longo94, David96,DasKod}, encoding the finite dimensional Hopf $\mathrm{C}^*$-algebras (also called Kac algebras \cite{IzuKos}). The set of maximal ones in this class can be seen as a very close extension of the set of prime numbers, because the only known ones are those coming from groups, which then must be cyclic of prime order, and it is an open problem whether there are other ones \cite[Problem 4.12]{Pal1}. The intermediates \cite{Bak,Lan} of a finite quantum group subfactor planar algebra correspond to the left coideal $*$-subalgebras of the Kac algebra $K$ \cite{ILP}, and the normal left coideal $*$-subalgebras correspond to the fusion subcategories of $\Rep(K)$ \cite{Bur}. A fusion category without nontrivial proper fusion subcategory is called \emph{simple}. Thus, a necessary condition for a subfactor planar algebra of this class to be maximal is that its fusion category be simple. Moreover, $\Rep(K)$ is a unitary integral fusion category. The Grothendieck ring of a fusion category is a fusion ring \cite{Lus}. So we are looking for simple integral fusion rings, not coming from a group, but with a unitary categorification. More generally, a fusion category which can be ``cooked up'' from group theory is called \emph{weakly-group theoretical} \cite{ENO11}, which  means that (up to equivalence) it has the same Drinfeld center as one coming from a sequence of group extensions from $\VVec$. But a weakly group-theoretical simple fusion category (over the complex field) is precisely Grothendieck equivalent to $\Rep(G)$ with $G$ a finite simple group \cite[Proposition 9.11]{ENO11}. So we are looking for integral fusion categories which are not weakly group-theoretical; their existence is one of the most famous open problem of fusion category theory \cite[Question 2]{ENO11}. Finally, the extension of Kaplansky's 6th conjecture \cite{Kap75} to every complex fusion category $\mathcal{C}$ \cite[Question 1]{ENO11} states that $\frac{\mathrm{FPdim}(\mathcal{C})}{\mathrm{FPdim}(x)}$ is an algebraic integer, for every simple object $x$ of $\mathcal{C}$ (property called \emph{Frobenius type}).

The paper \cite{LPW20} applies the quantum Schur product theorem \cite[Theorem 4.1]{Liuex} of quantum Fourier analysis \cite{QFA} as a criterion of unitary categorification, called Schur product criterion (see Theorem \ref{thm:schur}), and classifies all the simple integral fusion rings of Frobenius type up to rank $8$, $\FPdim$ less than $4080$ (and $\neq p^aq^b, pqr$, by \cite{ENO11}). There are $4$ group-like ones given by $\PSL(2,q)$ with $q \in \{4,7,9,11 \}$, $28$ ones excluded from unitary categorification by Schur product criterion, and exactly $2$ other ones, denoted $\mathcal{F}_{210}$ and  $\mathcal{F}_{660}$ (indexed by their $\FPdim$). The second one is excluded from any categorification by the zero spectrum criterion (see Theorem \ref{thm:zerocrit}) and also the one spectrum criterion (see Theorem \ref{thm:onecrit}), whereas the first one passes all the criteria known to us (listed in \S\ref{sec:criteria}), and we will now see the conceptual reason why. The family of finite simple groups of Lie type $\PSL(2,q)$, with $q$ prime-power, admits a generic character table depending on whether $q$ is even, $q \equiv -1$ or $1 \mod 4$. These character tables, written in \S\ref{sec:tab}, interpolate directly into formal tables for every integer $q \ge 2$. The use of  a Verlinde-like formula (Corollary \ref{coro:Verlinde}) on these interpolated tables provides rules (displayed in Theorems \ref{thm:q=0(2)}, \ref{thm:q=-1(4)} and \ref{thm:q=1(4)}, but computed in \S\ref{sec:cal}), which turn out to be those of fusion rings (Theorem \ref{thm:general}) and which have these formal tables as eigentables (Definition \ref{def:eigentable}), by reconstruction Theorem \ref{thm:TableFusion}. These \emph{interpolated} fusion rings $(\mathcal{R}_q)$ are integral, simple for $q \ge 4$, and non group-like when $q$ is not a prime-power. They inherit all the good arithmetical properties of the group-like ones, more precisely \S\ref{sec:criteria} proves that they pass all the known criteria. The smallest non prime-power is $q=6$, and it turns out that $\mathcal{R}_6 = \mathcal{F}_{210}$.

Here are observations on the fusion matrices of $\mathcal{R}_q$ (where the \emph{multiplicity} is the maximal entry): 
\begin{itemize}
\item for $q$ even: all self-adjoint, the non-trivial ones are of multiplicity $1$ or $2$,
\item for $q \equiv -1 \mod 4$: two non self-adjoint ones, the non-trivial ones are of multiplicity $1$, $2$ or $3$,
\item for $q \equiv 1 \mod 4$:  all self-adjoint, the non-trivial ones are of multiplicity $2$ or $3$ (except $q=5$).
\end{itemize}
This led to an interesting group-theoretical result: let $G$ be a non-abelian finite simple group, then the Grothendieck ring of $\Rep(G)$ is of multiplicity $\le 3$ if and only if $G$ is isomorphic to some $\PSL(2,q)$ with $q$ prime-power (see \cite{Pal4} for more details). Moreover, the family $(\PSL(2,q))$ is the only one, among the usual infinite families of non-abelian finite simple groups, where the fusion multiplicity is bounded above (because it is the only family of Lie type of rank $1$). This makes this family quite special, in the sense of fusion friendly. 

\begin{question}
Is there a non prime-power $q \ge 2$ for which the interpolated fusion ring $\mathcal{R}_q$ is categorifiable?
\end{question}

P. Etingof observed \cite{eti20} that it is still premature to believe in a positive answer of above question, by pointing out that an interpolated family $(\mathcal{T}_q)$ also exists from $\Rep(\mathbb{F}_q \rtimes \mathbb{F}_q^{\star})$, whereas by \cite[Corollary 7.4]{EGO04} there is a complex categorification (if and) only if $q$ is a prime-power. This family also pass all the criteria (proved in \S\ref{EtiMod}). Now, the category $\Rep(\mathbb{F}_q \rtimes \mathbb{F}_q^{\star})$ is just the near-group category \cite{Izu} $G+m$ with $G=C_{q-1}$ (cyclic group) and $m=|G|-1 = q-2$, but in general for all finite group $G$ and all $m \ge 0$, the corresponding fusion bialgebra $\mathcal{F}(G,m)$, or fusion ring when $m$ is also an integer, admits a \emph{subfactorization} (see \cite[Definition 2.26]{LPW20}) as the free product of the finite group and the Temperley-Lieb-Jones subfactor planar algebras (see \cite[\S2.4]{Liuex}), denoted  $\mathcal{S}^G * \mathcal{TLJ}_{\delta}$, with $\delta$ given by $\FPdim(\mathcal{F}(G,m)) = \delta^2 |G|$, so in particular, if $m=|G|-1$ (as above) then $\delta^2 = |G|+1$. So, Etingof's interpolated fusion ring $\mathcal{T}_q$ is categorifiable if and only if $q$ is a prime-power, but subfactorizable for all $q \ge 2$.

\begin{question}
Are the interpolated fusion rings $\mathcal{R}_q$ subfactorizable for all $q \ge 2$?
\end{question}

The paper \cite{LPR1} already excluded $\mathcal{R}_6 = \mathcal{F}_{210}$ from a categorification in characteristic zero. It was also excluded in positive characteristic, assuming pivotal (it is open otherwise). The problem is open for every non prime-power $q \neq 6$ (in every characteristic). To solve this problem we can try to generalize the localization strategy initiated in \cite{LPR1} and/or compute the generic F-symbols of the symmetric fusion category $\Rep(\PSL(2,q))$, $q$ prime-power, and see how and where they can be interpolated in the non prime-power case (non-symmetric, so some adjustments may be required). Note that a notion of interpolated tensor categories already appeared in the literature (see \cite{De3} and \cite[\S9.12]{EGNO15}), but of quite different nature. In fact Deligne interpolated the family of integral fusion categories $(\Rep(S_n))_{n \in \mathbb{Z}_{>0}}$ to the family of tensor categories $(\Rep(S_t))_{t \in \mathbb{C}}$, idem for some families of Lie type, but when $t$ is not an integer then they are neither integral nor fusion (in fact not even of subexponential growth).

Next \S\ref{sec:simple} proves general results on simple fusion categories (over $\mathbb{C}$). It proves that a non-pointed simple fusion category is non-braided if and only if its Drinfeld center is simple. In particular, a simple fusion category Morita equivalent to a simple fusion category with a noncommutative Grothendieck ring, is non-braided and has a simple Drinfeld center. This section also proves that every simple integral fusion category is weakly group-theoretical if and only if every simple integral modular fusion category is pointed.  We checked by computer assistance that every simple integral modular fusion category of rank up to $12$ is pointed, see \cite{ABPP}.
\begin{question}
Is there a non-pointed simple integral modular fusion category?
\end{question}
Finally, \S\ref{sec:braid} proves that if the interpolated fusion ring $\mathcal{R}_q$ admits a complex categorification, for some $q$ non prime-power, then it cannot be braided, and so its Drinfeld center must be simple.

In this paper, we discuss the interpolated fusion rings for $\Rep(\PSL(2,q))$, but it is natural to expect that, for every family of Lie type, the representation rings admit such interpolation (as suggested by the data in \cite{lub}). 

\begin{ac}
Thanks to Pavel Etingof, Dmitri Nikshych, Victor Ostrik and Arthur Jaffe for their encouraging interest on this work. Thanks to Pavel for his moderation. Thanks to Victor for the detailed proof of Theorem \ref{simpleZ}, and to Dmitri for his explanation about central functor. Thanks to Sebastian Burciu and Yilong Wang for useful discussions. The reference \cite[Example 2.1]{Na} was pointed out to us by Sebastian. Thanks to Jean Michel and Frank L\"ubeck, Jean led us to Frank who helped us with GAP.  The first author was supported by 2020YFA0713000 from NKPs and Grant 100301004 from Tsinghua University, the second one by BIMSA Start-up Research Fund and Foreign Youth Talent Program from the Ministry of Sciences and Technology of China, and the thrid one by Grant TRT 0159, ARO Grants W911NF-19-1-0302 and W911NF-20-1-0082.
\end{ac}

\tableofcontents

\section{Preliminaries}

This section recalls some basic notions about fusion rings (we refer to \cite[Chapter 3]{EGNO15}). We then show that the data of a commutative fusion ring is equivalent to the data of a formal table satisfying some assumptions. 

A \emph{fusion ring} $\mathcal{F}$ of \emph{rank} $r$ is defined by a finite set $\mathcal{B}=\{b_1, \dots, b_r\}$ called its \emph{basis}, and \emph{fusion rules} on $\mathbb{Z} \mathcal{B}$: $$ b_i b_j = \sum_{k=1}^r N_{i,j}^k b_k$$ with $N_{i,j}^k \in \mathbb{Z}_{\ge 0}$, satisfying axioms slightly extending the group axioms:
\begin{itemize}
\item \emph{Associativity.} $b_i (b_j b_k) = (b_i b_j) b_k $,
\item \emph{Neutral.} $b_1 b_i = b_i b_1 = b_i$,
\item \emph{Dual.} $\forall i \  \exists!i^* $ with $N_{i,k}^{1} = N_{k,i}^{1} = \delta_{i^*,k}$,
\item \emph{Frobenius reciprocity.} $N_{ij}^k = N_{i^*k}^j = N_{kj^*}^i$.
\end{itemize}
The dual structure $*$ induces an antihomomorphism of algebra, providing a structure of $*$-algebra on $\mathbb{C}\mathcal{B}$ defined by $b_i^* = b_{i^*}$, and by Frobenius-Perron theorem, there is a unique $*$-homomorphism $d:\mathbb{C}\mathcal{B} \to \mathbb{C}$ with $d(\mathcal{B}) \subset (0,\infty)$. The number $\FPdim(b_i):=d(b_i)$ is called the \emph{Frobenius-Perron dimension} of $b_i$, and $\FPdim(\mathcal{F}):=\sum_i d(b_i)^2$ is the Frobenius-Perron dimension of $\mathcal{F}$. The ordered sequence $(\FPdim(b))_{b \in \mathcal{B}}$ is called the \emph{type} of $\mathcal{F}$. The fusion ring $\mathcal{F}$ is called of \emph{Frobenius type} if  $\frac{\mathrm{FPdim}(\mathcal{F})}{\FPdim(b_i)}$ is an algebraic integer for all $i$. It is called \emph{integral} if $\FPdim(b_i)$ is an integer for all $i$. It is \emph{commutative} if $b_i b_j = b_j b_i$ for all $i,j$. The fusion matrices $M_i:= (N_{i,j}^k)_{k,j}$ provides a representation of $\mathcal{F}$ because $M_i M_j = \sum_k N_{ij}^k M_k$, moreover $M_i^* = M_{i^*}$ and $\FPdim(b_i)=\Vert M_i \Vert$. In the commutative case, $M_iM_j = M_j M_i$ for all $i,j$, in particular $M_iM_i^* = M_i^* M_i$, in other words, the fusion matrices are pairwise commutative and normal, so are simultaneously diagonalizable.

\begin{definition}[Eigentable] \label{def:eigentable}
Let $\mathcal{F}$ be a commutative fusion ring. Let $(M_i)$ be its fusion matrices, and let $(D_i) = (diag(\lambda_{i,j}))$ be their simultaneous diagonalization. The \emph{eigentable} of $\mathcal{F}$ is the table given by $(\lambda_{i,j})$.
\end{definition}

Note that the eigentable of the Grothendieck ring of $\Rep(G)$, with $G$ finite group, is the character table of $G$.


\begin{theorem}[Reconstruction] \label{thm:TableFusion}
Let $(\lambda_{i,j})$ be a formal $r \times r$ table. Consider
\begin{itemize}
\item the space of functions from $\{1, \dots, r \}$ to $\mathbb{C}$ with some inner product $ \langle f , g \rangle $,
\item the functions $(\lambda_i)$ defined by $\lambda_i(j) = \lambda_{i,j}$,
\item the pointwise multiplication $(fg)(i) = f(i) g(i)$, 
\item the multiplication operator $M_f: g \mapsto fg$, and $M_i:=M_{\lambda_i}$,
\end{itemize}
and assume that 
\begin{itemize}
\item $\langle \lambda_i , \lambda_j \rangle = \delta_{i,j}$,
\item for all $i$ there is $j$ (automatically unique, denoted $i^*$) such that $M_i^* = M_j$,
\item $M_1$ is the identity,
\item  $ N_{i,j}^k:= \langle \lambda_i \lambda_j, \lambda_k \rangle$ is a nonnegative integer for all $i,j,k$.
\end{itemize}
Then $(N_{i,j}^k)$ are the structure constants of a commutative fusion ring and $(\lambda_{i,j})$ is its eigentable. Moreover, every eigentable of a commutative fusion ring satisfies all the assumptions above.
\end{theorem} 

\begin{proof}
By assumption, the set $\{ \lambda_i \ | \ i=1,\dots,r \}$ is an orthonormal basis of the $r$-dimensional space it generates, so this space must be the $r$-dimensional space of functions from $\{1, \dots, r \}$ to $\mathbb{C}$. It follows that for all $i,j$, $\lambda_i \lambda_j$ decomposes into a linear sum of $\lambda_k$, and so by construction $$\lambda_i \lambda_j = \sum_k N_{i,j}^k \lambda_k.$$ 
We can now show that $(N_{i,j}^k)$ are the structure constants of a commutative fusion ring: 
\begin{itemize}
\item Commutativity. $\lambda_i \lambda_j = \lambda_j \lambda_i$ because the multiplication is pointwise and $\mathbb{C}$ is commutative,
\item Associativity. $(\lambda_i \lambda_j)\lambda_k = \lambda_i (\lambda_j \lambda_k)$ because the multiplication is pointwise and $\mathbb{C}$ is associative,
\item Frobenius reciprocity. $N_{i,j}^k  = \langle M_i \lambda_j , \lambda_k \rangle  = \langle  \lambda_j , M_i^* \lambda_k \rangle = \overline{\langle  M_i^* \lambda_k  , \lambda_j \rangle} = \overline{N_{i^*,k}^j} = N_{i^*,k}^j = \dots = N_{k,j^*}^i$,
\item Neutral. $N_{1,i}^k = N_{i,1}^k = \langle \lambda_i \lambda_1, \lambda_k \rangle = \langle \lambda_i , \lambda_k \rangle = \delta_{i,k}$,
\item Dual. $N_{i,j}^1 = N_{i^*,1}^j = \delta_{i^*,j}$.
\end{itemize}

Finally, consider the function $\delta_i: j \mapsto \delta_{i,j}$. Then $M_i \delta_j = \lambda_{i,j} \delta_j $. It follows that $(diag(\lambda_{i,j}))$ is a simultaneous diagonalization of $(M_i)$. So $(\lambda_{i,j})$ is an eigentable of the commutative fusion ring defined by $(N_{i,j}^k)$.

Reciprocally, let $(\lambda_{i,j})$ be the eigentable of a commutative fusion ring. Define  $(\lambda_i)$ as above and the inner product by $\langle \lambda_i , \lambda_j \rangle = \delta_{i,j}$. Then all the assumptions in the statement of the theorem follows easily.
\end{proof}

\begin{definition}[Formal codegrees] \label{def:codeg}
The \emph{formal codegrees} of a commutative fusion ring $\mathcal{F}$ with  eigentable $(\lambda_{i,j})$ are 
$$\mathfrak{c}_j:=\sum_i \vert \lambda_{i,j} \vert ^2.$$
\end{definition}

\begin{theorem}[Schur orthogonality relations] \label{thm:SchurOrtho}
Following the notations of Definition \ref{def:codeg}:
$$\sum_i \lambda_{i,j} \overline{\lambda_{i,j'}} = \delta_{j,j'} \mathfrak{c}_{j} \ \text{ and } \  \sum_j \frac{1}{\mathfrak{c}_j} \lambda_{i,j} \overline{\lambda_{i',j}} = \delta_{i,i'}.$$
\end{theorem}
\begin{proof}
Observe that the map $i \mapsto \lambda_{i,j}$ induces a character on $\mathcal{F}$, so the first relation follows from \cite[Lemma 8.14.1]{EGNO15}. Next, let $U$ be the matrix $(\frac{1}{\sqrt{\mathfrak{c}_j}}\overline{\lambda_{i,j}})$. The first relation means that $U^*U = id$, i.e. $U$ is an isometry. But in the finite dimensional case, an isometry is unitary, so $UU^* = id$ also, which means that:
$$  \sum_j \frac{\overline{\lambda_{i,j}}}{\sqrt{\mathfrak{c}_j}} \frac{\lambda_{i',j}}{\sqrt{\mathfrak{c}_j}} = \delta_{i,i'}.  $$ The second relation follows.
\end{proof}

\begin{corollary}[Inner product] \label{coro:inner}
In Theorem \ref{thm:TableFusion}, the inner product can always be taken of the form
$$\langle f , g \rangle := \sum_s \frac{1}{\mathfrak{c}_s} f(s) \overline{g(s)}.$$
\end{corollary}  

\begin{corollary}[Verlinde-like formula] \label{coro:Verlinde}
Let $\mathcal{F}$ be a commutative fusion ring with basis $(b_i)$, fusion rules $(N_{i,j}^k)$, eigentable $(\lambda_{i,j})$ and formal codegrees $(\mathfrak{c}_j)$. Then
$$ N_{i,j}^k  =  \langle b_i b_j, b_k \rangle = \sum_s \frac{1}{\mathfrak{c}_s}\lambda_{i,s} \lambda_{j,s} \overline{\lambda_{k,s}} = \sum_s \frac{\lambda_{i,s} \lambda_{j,s} \overline{\lambda_{k,s}}}{\sum_l \vert \lambda_{l,s} \vert ^2}. $$
\end{corollary}  

\begin{remark} \label{rem:orb-stab}
Let $G$ be a finite group, let $g$ be an element of $G$, let $Cl(g)$ be its conjugacy class and let $C_G(g)$ be its centralizer. By the orbit-stabilizer theorem $|G| =  |Cl(g)| \times |C_G(g)| $. By \cite[Theorem 2.18]{Isa}, the formal codegree is $|C_G(g)|$, whereas (by definition) the class size is $|Cl(g)|$. 
\end{remark}
Now the conjugacy classes form a partition of $G$, so $\sum_i |G|/\mathfrak{c}_i = |G|$, which provides the Egyptian fraction $\sum_i 1/\mathfrak{c}_i = 1$. It generalizes as follows:

\begin{corollary} \label{coro:egy}
Let $\mathcal{F}$ be a commutative fusion ring and let $(\mathfrak{c}_i)$ be its formal codegrees. Then $\sum_i 1/\mathfrak{c}_i = 1$.
\end{corollary}  
\begin{proof}
Apply Corollary \ref{coro:inner} to $f=g=\lambda_1$.
\end{proof}

In particular, $\FPdim(\mathcal{F}) = \mathfrak{c}_1 = \sum_i \mathfrak{c}_1/\mathfrak{c}_i$, so we may see $(\mathfrak{c}_1/\mathfrak{c}_i)$ as \emph{formal class sizes} (and see \S\ref{sub:drinfeld}).

\section{Generic character table of $\PSL(2,q)$} \label{sec:tab}
The generic character table for $\PSL(2,q)$ requires to consider three cases depending on $q \mod 4$. We will use the data provided by \cite{GAP4} as a reference for these tables, but for more usual references, see for example \cite[\S 5.2]{FuHa}, \cite[\S 12.5]{DiMi} or \cite{Bon}. Let us briefly recall that in the character table of a finite group, each line is labelled by an irreducible complex character $\chi$, each column is labelled by a conjugacy class $K$, and the entry at $(\chi,K)$ is just $\chi(g)$ with $g \in K$. The class size of $K$ is its order. In the tables displayed below some irreducible characters are grouped together in a single line and identified by a parameter $c$ (replacing the parameter `i' or its half in GAP), idem for some conjugacy classes with a parameter $k$. Let $\zeta_n$ be the root of unity $\exp(\frac{2\pi i}{n})$, where $i$ denotes the imaginary unit.

\begin{ttable} \label{table:q=0(2)}
 Here is the generic character table of $\PSL(2,q)$ with $q$ even:
$$
\begin{tabular}{ |c||c|c|c|c|c| } 
\hline
\backslashbox{\text{charparam }$c$}{\text{classparam }$k$} & $\{1\}$ & $\{1\}$ & $\{1, \dots, \frac{q-2}{2}\}$ & $\{1, \dots, \frac{q}{2}\}$ \\   \hline \hline
 $\{1\}$ & $1$ & $1$ & $1$ & $1$  \\ 		 
 $\{1, \dots, \frac{q}{2}\}$ & $q-1$ & $-1$ & $0$ & $-\zeta_{q+1}^{kc} -\zeta_{q+1}^{-kc}$  \\ 		  
 $\{1\}$ & $q$ & $0$ & $1$ & $-1$    \\ 		  
 $\{1, \dots, \frac{q-2}{2}\}$ & $q+1$ & $1$ & $\zeta_{q-1}^{kc} + \zeta_{q-1}^{-kc}$ & $0$    \\  \hline \hline	  
  \text{ class size } & $1$ & $q^2-1$ & $q(q+1)$  & $q(q-1)$ \\    \hline
\end{tabular}
$$
\end{ttable}
\begin{proof}
Here is the GAP code providing these data:

\begin{verbatim}
gap> Print(CharacterTableFromLibrary("SL2even"));
\end{verbatim}
Recall that for $q$ even, $\PSL(2,q) = \SL(2,q)$.
\end{proof}

\begin{ttable} \label{table:q=-1(4)}
 Here is the generic character table of $\PSL(2,q)$ with $q \equiv -1 \mod 4$:
 
$$
\begin{tabular}{ |c||c|c|c|c|c|c|c| } 
\hline
\backslashbox{\text{charparam }$c$}{\text{classparam }$k$} & $\{1\}$ & $\{1,2\}$ & $\{1, \dots, \frac{q-3}{4}\}$ & $\{1, \dots, \frac{q-3}{4}\}$  & $\{\frac{q+1}{4}\}$ \\   \hline \hline
$\{1\}$ & $1$ & $1$ & $1$ & $1$ & $1$ \\ 
$\{1,2\}$ & $\frac{q-1}{2}$&$\frac{-1+i(-1)^{k+c}\sqrt{q}}{2}$&$0$&$(-1)^{k+1}$&$(-1)^{k+1}$ \\ 
$\{1, \dots, \frac{q-3}{4}\}$ & $q-1$ & $-1$ & $0$  & $-\zeta_{q+1}^{2kc}-\zeta_{q+1}^{-2kc}$ & $-2(-1)^c$ \\
$\{1\}$ & $q$ & $0$ & $1$ & $-1$ & $-1$  \\ 		  
$\{1, \dots, \frac{q-3}{4}\}$ & $q+1$ & $1$ & $\zeta_{q-1}^{2kc}+\zeta_{q-1}^{-2kc}$   & $0$ & $0$ \\ 
 \hline \hline	  
  \text{ class size } & $1$ & $\frac{q^2-1}{2}$ & $q(q+1)$& $q(q-1)$ & $\frac{q(q-1)}{2}$ \\    \hline
\end{tabular}
$$
\end{ttable}

\begin{proof}
Here is the GAP code providing these data:

\begin{verbatim}
gap> Print(CharacterTableFromLibrary("PSL2odd"));
\end{verbatim}
Above table is a slight simplification of the one provided by GAP in the sense that the line labelled $\{1,2\}$ is a grouping of two lines (idem for columns). The notation \verb_EB(q)_ used in GAP means $\frac{-1+i\sqrt{q}}{2}$ when $q \equiv -1 \mod 4$.
\end{proof}

\begin{ttable}  \label{table:q=1(4)}
 Here is the generic character table of $\PSL(2,q)$ with $q \equiv 1 \mod 4$:

$$
\begin{tabular}{ |c||c|c|c|c|c|c| } 
\hline
\backslashbox{\text{charparam }$c$}{\text{classparam }$k$} & $\{1\}$ & $\{1,2\}$ & $\{1, \dots, \frac{q-5}{4}\}$ & $\{\frac{q-1}{4}\}$ & $\{1, \dots, \frac{q-1}{4}\}$ \\   \hline \hline
 $\{1\}$ & $1$ & $1$ & $1$ & $1$ & $1$ \\ 	
 $\{1,2\}$&$\frac{q+1}{2}$&$\frac{1+(-1)^{k+c}\sqrt{q}}{2}$&$(-1)^k$&$(-1)^k$ & $0$ \\ 
 $\{1, \dots, \frac{q-1}{4}\}$ & $q-1$  & $-1$ & $0$ & $0$ & $-\zeta_{q+1}^{2kc}-\zeta_{q+1}^{-2kc}$ \\ 		  
 $\{1\}$ & $q$ & $0$ & $1$ & $1$ & $-1$  \\ 		  
 $\{1, \dots, \frac{q-5}{4}\}$ & $q+1$  & $1$ & $\zeta_{q-1}^{2kc}+\zeta_{q-1}^{-2kc}$   & $2(-1)^c$ & $0$ \\ 
  \hline \hline	  
  \text{ class size } & $1$ & $\frac{q^2-1}{2}$ & $q(q+1)$ & $\frac{q(q+1)}{2}$ & $q(q-1)$ \\    \hline
\end{tabular}
$$
\end{ttable}

\begin{proof}
Here is the GAP code providing these data:

\begin{verbatim}
gap> Print(CharacterTableFromLibrary("PSL2even"));
\end{verbatim}
Same simplification than for Table \ref{table:q=-1(4)}. The notation \verb_EB(q)_ used in GAP means $=\frac{-1+\sqrt{q}}{2}$ when $q \equiv 1 \mod 4$. Moreover $2(-1)^c$ simplifies the entry \verb_E(4)^i + E(4)^(-i)_ in GAP,
where the parameter `i' is $2c$.
\end{proof}

\section{Interpolated fusion rings} \label{sec:interpol}
The generic character table of $\PSL(2,q)$ as provided by GAP can be interpolated to $q$ non prime-power. In this section, we show that this interpolated formal table is in fact the character table of an interpolated ``$\PSL(2,q)$'', more precisely, we will prove that the application of the Verlinde-like formula (Corollary \ref{coro:Verlinde}) to these formal tables provides a family of fusion rings ($\mathcal{R}_q$) which are exactly the Grothendieck rings of $\Rep(\PSL(2,q))$ when $q$ is a prime-power, but otherwise, new non group-like simple integral fusion rings, whose eigentables are exactly these formal tables.

\begin{theorem} \label{thm:general}
The generic character table of $\PSL(2,q)$, interpolated to every integer $q \ge 2$, is the eigentable of an integral fusion ring, called the \emph{interpolated fusion ring} $\mathcal{R}_q$. Moreover, if $q \ge 4$ then it is simple, and if $q$ is not a prime-power then it is not group-like.
\end{theorem}
\begin{proof}
The rules displayed in Theorems \ref{thm:q=0(2)}, \ref{thm:q=-1(4)}, \ref{thm:q=1(4)} (and computed in \S \ref{sec:cal} from Tables \ref{table:q=0(2)}, \ref{table:q=-1(4)}, \ref{table:q=1(4)}) satisfy all the axioms of a fusion ring by applying Theorem \ref{thm:TableFusion} together with Propositions \ref{prop:ortho}, \ref{prop:ortho2}, \ref{prop:ortho3} (with the inner product as in Corollary \ref{coro:inner}), and admit these tables as eigentables, again by Theorem \ref{thm:TableFusion}. To show that $\mathcal{R}_q$ is simple if $q \ge 4$, just observe that every non-trivial simple object generates all ones, so that there is no proper non-trivial fusion subring. Finally, it is not group-like when $q$ is not a prime-power by Lemma \ref{lem:NoOverlap}.
\end{proof}


%

\begin{lemma} \label{lem:NoOverlap}
For $q$ non prime-power, the interpolated fusion ring $\mathcal{R}_q$ is not the Grothendieck ring of some $\Rep(G)$ with $G$ a finite group.
\end{lemma}
\begin{proof}
Assume that for $q$ non prime-power (and so $\ge 6$), there exists such a finite group $G$ (obviously non-abelian). The interpolated fusion ring $\mathcal{R}_q$ being simple, the Grothendieck ring of $\Rep(G)$ and $G$ itself must be so. So by the classification of finite simple groups, $G$ must be of Lie type, alternating or sporadic.   

Firstly, assume that $G$ is of Lie type, so defined over a finite field of characterisitc $p$. By \cite{hum}, the character values for the Steinberg representation (complex, irreducible and non-trivial) on an element $g$ equals, up to a sign, the order of a $p$-Sylow subgroup of the centralizer of $g$ if $p$ does not divide the order of $g$, and zero otherwise (in particular all integral). But it never fit with $\mathcal{R}_q$. Let us show why for $q$ even (the odd case is similar). Consider Table \ref{table:q=0(2)}. The Steinberg representation can neither be the one of degree $1$ (because it is non-trivial), nor the one of degree $q$ (because the degree is a character value and $q$ is not prime-power, whereas the order of $p$-Sylow is so). There remain the ones of degree $q \pm 1$. Assume it is of degree $q-1$ (the proof for degree $q+1$ is similar). Let $c \in \{1, \dots, \frac{q}{2}\}$ be the parameter of the Steinberg character, and consider an element $g$ with character value $-\zeta_{q+1}^{kc} -\zeta_{q+1}^{-kc}$. Its class size is $q(q-1)$, so the order of its centralizer is $q+1$, which is coprime with $q-1$, but the order of $g$ divides the order of its centralizer and $q-1$ must be a power of $p$ (by assumption), so $p$ does not divide the order of $g$. Then its character value is the order of a $p$-Sylow subgroup of the centralizer, up to a sign. It follows that $\zeta_{q+1}^{kc} +\zeta_{q+1}^{-kc} = 2\cos(\frac{2\pi kc}{q+1}) = \pm 1$, i.e. $\frac{kc}{q+1} \equiv \pm \frac{1}{3}$ or $ \pm \frac{1}{6} \mod 1$, for all $k \in \{1, \dots, \frac{q}{2}\}$. Take $k=1$ and multiply by $3$, we get that $\frac{3c}{q+1} \equiv 0$ or $\frac{1}{2} \mod 1$. But $q \ge 6$, so we can directly take $k=3$ to get $\frac{3c}{q+1} \equiv \pm \frac{1}{3}$ or $ \pm \frac{1}{6} \mod 1$, contradiction.

Secondly, assume that $G = A_n$. It is well-known that for $n \ge 7$, there is a single non-trivial irreducible representation of smallest degree $n-1$ (see \cite{gill} for more details about this folklore result). This contradicts with the type of $\mathcal{R}_q$ where the smallest non-trivial degree is not unique. So $n = 5,6$, but $A_5 \simeq \PSL(2,4)$ and $A_6 \simeq \PSL(2,9)$, both of Lie type.

Finally, assume that $G$ is sporadic. The following GAP computation
\begin{verbatim}
gap> L:=["M11","M12","M22","M23","M24","J1","J2","J3","J4","Co1","Co2","Co3",
"Fi22","Fi23","Fi24'","Suz","HS","McL","He","HN","Th","B","M","ON","Ly","Ru","Tits"];;
gap> Minimum(List(L, N -> Length(CharacterDegrees(CharacterTable(N)))));
7
\end{verbatim}
\noindent shows that the number of different character degrees is greater than $6$, whereas it is less than $6$ for $\mathcal{R}_q$, contradiction.
\end{proof}

Let $\mathcal{F}$ be a fusion ring with basis $\mathcal{B} = \{b_1, \dots, b_r\}$. As mentioned above, the type of $\mathcal{F}$ is the sequence $(\FPdim(b_i))$. Let $1=d_1 < \dots < d_s$ such that the set $\{d_1, \dots, d_s\} = \{\FPdim(b_1), \dots, \FPdim(b_r)\}$. Note that $s \le r$ without equality in general because two elements in the basis can have the same $\FPdim$. Let $x_{d_i,1}, \dots, x_{d_i,m_i}$ be all the elements in the basis with $\FPdim = d_i$.  The list $[[d_1,m_1],\dots,[d_s,m_s]]$ is an alternative way to write the type of $\mathcal{F}$, and it will be called \emph{type} as well. In the rest of the paper, $x_{d,c}$ denotes the c-th basis element with $\FPdim = d$.

\begin{theorem} \label{thm:q=0(2)} 
If $q$ is even, the interpolated fusion ring $\mathcal{R}_q$ has rank $q+1$, $\FPdim \ q(q^2-1)$, type $$[[1,1],[q-1,q/2],[q,1],[q+1,(q-2)/2]],$$ and fusion rules
\begin{align*}
 x_{q-1,c_1} x_{q-1,c_2} & = \delta_{c_1,c_2}x_{1,1} + \hspace*{-.8cm} \sum_{\substack{c_3 \text{ such that } \\ c_1+c_2+c_3 \neq q+1 \\ \text{ and } 2\Max(c_1,c_2,c_3)}} \hspace*{-.8cm} x_{q-1,c_3} + (1- \delta_{c_1,c_2})x_{q,1} + \sum_{c_3} x_{q+1,c_3}, \\
 x_{q-1,c_1} x_{q,1} & = \sum_{c_2} (1- \delta_{c_1,c_2})x_{q-1,c_2} + x_{q,1} + \sum_{c_2} x_{q+1,c_2}, \\
 x_{q-1,c_1} x_{q+1,c_2} & = \sum_{c_3} x_{q-1,c_3} + x_{q,1} + \sum_{c_3} x_{q+1,c_3}, \\
 x_{q,1} x_{q,1} & = x_{1,1} + \sum_c x_{q-1,c} + x_{q,1} + \sum_c x_{q+1,c}, \\
 x_{q,1} x_{q+1,c_1} & =  \sum_{c_2} x_{q-1,c_2} + x_{q,1} + \sum_{c_2} (1 + \delta_{c_1,c_2}) x_{q+1,c_2}, \\
 x_{q+1,c_1} x_{q+1,c_2} & =  \delta_{c_1,c_2}x_{1,1} + \sum_{c_3} x_{q-1,c_3}  + (1 + \delta_{c_1,c_2})x_{q,1} \hspace*{.2cm} + \hspace*{-.8cm} \sum_{\substack{c_3 \text{ such that } \\ c_1+c_2+c_3 \neq q-1 \\ \text{ and } 2\Max(c_1,c_2,c_3)}} \hspace*{-.8cm}  x_{q+1,c_3} \hspace*{.2cm} + \hspace*{-.6cm}  \sum_{\substack{c_3 \text{ such that } \\ c_1+c_2+c_3 = q-1 \\ \text{ or } 2\Max(c_1,c_2,c_3)}} \hspace*{-.8cm}  2x_{q+1,c_3}. 
\end{align*}
The other rules follow by commutativity.
\end{theorem}
\begin{proof}
The fusion rules follow from Table \ref{table:q=0(2)} and Lemmas \ref{lem:fullcomput} to \ref{lem:last}, together with the Verlinde-like formula in Corollary \ref{coro:Verlinde}.
\end{proof}

We deduce the following unexpected combinatorial result: 
\begin{corollary} Consider an even integer $q \ge 2$. For $1 \le c_i \le q/2$, the following map is invariant by permutation. $$ (c_1,c_2,c_3,c_4) \mapsto \delta_{c_1,c_2}\delta_{c_3,c_4} - \# \{ |q+1-2|x||, \ x \in \{c_1+c_2, c_3+c_4, c_1-c_2, c_3-c_4\} \}$$ 
\end{corollary}
\begin{proof}
The following equality holds by associativity $$\langle (x_{q-1,c_1} x_{q-1,c_2}) x_{q-1,c_3},  x_{q-1,c_4} \rangle = \langle x_{q-1,c_1} (x_{q-1,c_2} x_{q-1,c_3}), x_{q-1,c_4} \rangle,$$
now, by computing the LHS and RHS independently and directly by the fusion rules of $\mathcal{R}_q$, the equality of the outcome provides the result after straighforward reformulations.
\end{proof}

\begin{remark}
Other such combinatorial results can be obtained by considering other associativity equalities.
\end{remark}

The smallest even non prime-power is $q=6$, and $\mathcal{R}_{6}$ has rank $7$, $\FPdim \ 210$, type $$[[1,1],[5,3],[6,1],[7,2]],$$  and fusion matrices (respectively corresponding to the basis elements $x_{1,1}$, $x_{5,1}$, $x_{5,2}$, $x_{5,3}$, $x_{6,1}$, $x_{7,1}$, $x_{7,2}$):
$$
\begin{smallmatrix} 1&0&0&0&0&0&0 \\ 0&1&0&0&0&0&0 \\ 0&0&1&0&0&0&0 \\ 0&0&0&1&0&0&0 \\ 0&0&0&0&1&0&0 \\ 0&0&0&0&0&1&0 \\ 0&0&0&0&0&0&1 \end{smallmatrix} , \ \begin{smallmatrix} 0&1&0&0&0&0&0 \\ 1&1&0&1&0&1&1 \\ 0&0&1&0&1&1&1 \\ 0&1&0&0&1&1&1 \\ 0&0&1&1&1&1&1 \\ 0&1&1&1&1&1&1 \\ 0&1&1&1&1&1&1 \end{smallmatrix} , \ \begin{smallmatrix} 0&0&1&0&0&0&0 \\ 0&0&1&0&1&1&1 \\ 1&1&1&0&0&1&1 \\ 0&0&0&1&1&1&1 \\ 0&1&0&1&1&1&1 \\ 0&1&1&1&1&1&1 \\ 0&1&1&1&1&1&1 \end{smallmatrix} , \ \begin{smallmatrix} 0&0&0&1&0&0&0 \\ 0&1&0&0&1&1&1 \\ 0&0&0&1&1&1&1 \\ 1&0&1&1&0&1&1 \\ 0&1&1&0&1&1&1 \\ 0&1&1&1&1&1&1 \\ 0&1&1&1&1&1&1 \end{smallmatrix} , \ \begin{smallmatrix} 0&0&0&0&1&0&0 \\ 0&0&1&1&1&1&1 \\ 0&1&0&1&1&1&1 \\ 0&1&1&0&1&1&1 \\ 1&1&1&1&1&1&1 \\ 0&1&1&1&1&2&1 \\ 0&1&1&1&1&1&2 \end{smallmatrix} , \ \begin{smallmatrix} 0&0&0&0&0&1&0 \\ 0&1&1&1&1&1&1 \\ 0&1&1&1&1&1&1 \\ 0&1&1&1&1&1&1 \\ 0&1&1&1&1&2&1 \\ 1&1&1&1&2&1&2 \\ 0&1&1&1&1&2&2 \end{smallmatrix} , \ \begin{smallmatrix} 0&0&0&0&0&0&1 \\ 0&1&1&1&1&1&1 \\ 0&1&1&1&1&1&1 \\ 0&1&1&1&1&1&1 \\ 0&1&1&1&1&1&2 \\ 0&1&1&1&1&2&2 \\ 1&1&1&1&2&2&1 \end{smallmatrix} 
$$

\begin{theorem} \label{thm:q=-1(4)}
If $q \equiv -1 \mod 4$, the interpolated fusion ring $\mathcal{R}_q$ has rank $(q+5)/2$, $\FPdim \ q(q^2-1)/2$, type $$[[1,1],[(q-1)/2,2],[q-1,(q-3)/4],[q,1],[q+1,(q-3)/4]],$$ and fusion rules
\begin{align*}
x_{\frac{q-1}{2},c_1} x_{\frac{q-1}{2},c_2} & = (1-\delta_{c_1,c_2})x_{1,1} + \delta_{c_1,c_2}\sum_{c_3}(1-\delta_{c_1,c_3}) x_{\frac{q-1}{2},c_3} + \delta_{c_1,c_2}\sum_{c_3}x_{q-1,c_3} + (1 - \delta_{c_1,c_2})\sum_{c_3}x_{q+1,c_3}, \\
x_{\frac{q-1}{2},c_1} x_{q-1,c_2} & = \sum_{c_3}(1-\delta_{c_1,c_3})x_{\frac{q-1}{2},c_3} + \sum_{c_3}(1-\delta_{c_2+c_3,\frac{q+1}{4}})x_{q-1,c_3} + x_{q,1} + \sum_{c_3} x_{q+1,c_3}, \\
x_{\frac{q-1}{2},c_1} x_{q,1} & = \sum_{c_2} x_{q-1,c_2} + x_{q,1} + \sum_{c_2} x_{q+1,c_2}, \\
x_{\frac{q-1}{2},c_1} x_{q+1,c_2} & = \sum_{c_3} \delta_{c_1,c_3}x_{\frac{q-1}{2},c_3} + \sum_{c_3}x_{q-1,c_3} + x_{q,1} + \sum_{c_3} x_{q+1,c_3}, \\
x_{q-1,c_1} x_{q-1,c_2} & = \delta_{c_1,c_2}x_{1,1} + (1-\delta_{c_1+c_2,\frac{q+1}{4}})\sum_{c_3}x_{\frac{q-1}{2},c_3} + \hspace{-.9cm} \sum_{\substack{c_3 \text{ such that } \\ c_1+c_2+c_3 = \frac{q+1}{2} \\ \text{ or } 2\Max(c_1,c_2,c_3)}} \hspace{-.8cm} x_{q-1,c_3} +  2 \hspace{-.9cm}  \sum_{\substack{c_3 \text{ such that } \\ c_1+c_2+c_3 \neq \frac{q+1}{2} \\ \text{ and } 2\Max(c_1,c_2,c_3)}} \hspace{-1cm} x_{q-1,c_3} +  (2-\delta_{c_1,c_2})x_{q,1} + 2\sum_{c_3} x_{q+1,c_3}, \\
x_{q-1,c_1} x_{q,1} & = \sum_{c_2} x_{\frac{q-1}{2},c_2} + \sum_{c_2} (2-\delta_{c_1,c_2})x_{q-1,c_2} + 2x_{q,1} + 2 \sum_{c_2} x_{q+1,c_2}, \\
x_{q-1,c_1} x_{q+1,c_2} & = \sum_{c_3} x_{\frac{q-1}{2},c_3}  + 2\sum_{c_3} x_{q-1,c_3} + 2x_{q,1} + 2\sum_{c_3} x_{q+1,c_3}, \\
x_{q,1} x_{q,1} & = x_{1,1} + \sum_{c_1} x_{\frac{q-1}{2},c_1} + 2\sum_{c_1} x_{q-1,c_1} + 2x_{q,1} + 2\sum_{c_1} x_{q+1,c_1}, \\
x_{q,1} x_{q+1,c_1} & = \sum_{c_2} x_{\frac{q-1}{2},c_2} + 2\sum_{c_2} x_{q-1,c_2} + 2x_{q,1} + \sum_{c_2}(2+\delta_{c_1,c_2})x_{q+1,c_2}, \\
x_{q+1,c_1} x_{q+1,c_2} & = \delta_{c_1,c_2}x_{1,1} + \sum_{c_3} x_{\frac{q-1}{2},c_3} + 2\sum_{c_3} x_{q-1,c_3} + (2+\delta_{c_1,c_2})x_{q,1} + \hspace{.4cm}  2 \hspace{-.9cm}  \sum_{\substack{c_3 \text{ such that } \\ c_1+c_2+c_3 \neq \frac{q-1}{2} \\ \text{ and } 2\Max(c_1,c_2,c_3)}} \hspace{-1cm} x_{q+1,c_3} + \hspace{.4cm}  3 \hspace{-.9cm}\sum_{\substack{c_3 \text{ such that } \\ c_1+c_2+c_3 = \frac{q-1}{2} \\ \text{ or } 2\Max(c_1,c_2,c_3)}} \hspace{-.8cm} x_{q+1,c_3}.
\end{align*}
\end{theorem}
\begin{proof}
The rules are given by Table \ref{table:q=-1(4)} and Lemmas \ref{lem:first2} to \ref{lem:last2}.
\end{proof}

The smallest non prime-power $q \equiv -1 \mod 4$ is $q=15$, and $\mathcal{R}_{15}$ has rank $10$, $\FPdim \ 1680$, type $$[[1,1],[7,2],[14,3],[15,1],[16,3]],$$ and fusion matrices:
$$
\begin{smallmatrix} 1&0&0&0&0&0&0&0&0&0 \\ 0&1&0&0&0&0&0&0&0&0 \\ 0&0&1&0&0&0&0&0&0&0 \\ 0&0&0&1&0&0&0&0&0&0 \\ 0&0&0&0&1&0&0&0&0&0 \\ 0&0&0&0&0&1&0&0&0&0 \\ 0&0&0&0&0&0&1&0&0&0 \\ 0&0&0&0&0&0&0&1&0&0 \\ 0&0&0&0&0&0&0&0&1&0 \\ 0&0&0&0&0&0&0&0&0&1 \end{smallmatrix} , \ \begin{smallmatrix} 0&1&0&0&0&0&0&0&0&0 \\ 0&0&1&1&1&1&0&0&0&0 \\ 1&0&0&0&0&0&0&1&1&1 \\ 0&0&1&1&1&0&1&1&1&1 \\ 0&0&1&1&0&1&1&1&1&1 \\ 0&0&1&0&1&1&1&1&1&1 \\ 0&0&0&1&1&1&1&1&1&1 \\ 0&1&0&1&1&1&1&1&1&1 \\ 0&1&0&1&1&1&1&1&1&1 \\ 0&1&0&1&1&1&1&1&1&1 \end{smallmatrix} , \ \begin{smallmatrix} 0&0&1&0&0&0&0&0&0&0 \\ 1&0&0&0&0&0&0&1&1&1 \\ 0&1&0&1&1&1&0&0&0&0 \\ 0&1&0&1&1&0&1&1&1&1 \\ 0&1&0&1&0&1&1&1&1&1 \\ 0&1&0&0&1&1&1&1&1&1 \\ 0&0&0&1&1&1&1&1&1&1 \\ 0&0&1&1&1&1&1&1&1&1 \\ 0&0&1&1&1&1&1&1&1&1 \\ 0&0&1&1&1&1&1&1&1&1 \end{smallmatrix} , \ \begin{smallmatrix} 0&0&0&1&0&0&0&0&0&0 \\ 0&0&1&1&1&0&1&1&1&1 \\ 0&1&0&1&1&0&1&1&1&1 \\ 1&1&1&2&1&2&1&2&2&2 \\ 0&1&1&1&2&1&2&2&2&2 \\ 0&0&0&2&1&2&2&2&2&2 \\ 0&1&1&1&2&2&2&2&2&2 \\ 0&1&1&2&2&2&2&2&2&2 \\ 0&1&1&2&2&2&2&2&2&2 \\ 0&1&1&2&2&2&2&2&2&2 \end{smallmatrix} , \ \begin{smallmatrix} 0&0&0&0&1&0&0&0&0&0 \\ 0&0&1&1&0&1&1&1&1&1 \\ 0&1&0&1&0&1&1&1&1&1 \\ 0&1&1&1&2&1&2&2&2&2 \\ 1&0&0&2&2&2&1&2&2&2 \\ 0&1&1&1&2&1&2&2&2&2 \\ 0&1&1&2&1&2&2&2&2&2 \\ 0&1&1&2&2&2&2&2&2&2 \\ 0&1&1&2&2&2&2&2&2&2 \\ 0&1&1&2&2&2&2&2&2&2 \end{smallmatrix},$$
\vspace*{.1cm}
$$\begin{smallmatrix} 0&0&0&0&0&1&0&0&0&0 \\ 0&0&1&0&1&1&1&1&1&1 \\ 0&1&0&0&1&1&1&1&1&1 \\ 0&0&0&2&1&2&2&2&2&2 \\ 0&1&1&1&2&1&2&2&2&2 \\ 1&1&1&2&1&2&1&2&2&2 \\ 0&1&1&2&2&1&2&2&2&2 \\ 0&1&1&2&2&2&2&2&2&2 \\ 0&1&1&2&2&2&2&2&2&2 \\ 0&1&1&2&2&2&2&2&2&2 \end{smallmatrix} , \ \begin{smallmatrix} 0&0&0&0&0&0&1&0&0&0 \\ 0&0&0&1&1&1&1&1&1&1 \\ 0&0&0&1&1&1&1&1&1&1 \\ 0&1&1&1&2&2&2&2&2&2 \\ 0&1&1&2&1&2&2&2&2&2 \\ 0&1&1&2&2&1&2&2&2&2 \\ 1&1&1&2&2&2&2&2&2&2 \\ 0&1&1&2&2&2&2&3&2&2 \\ 0&1&1&2&2&2&2&2&3&2 \\ 0&1&1&2&2&2&2&2&2&3 \end{smallmatrix} , \ \begin{smallmatrix} 0&0&0&0&0&0&0&1&0&0 \\ 0&1&0&1&1&1&1&1&1&1 \\ 0&0&1&1&1&1&1&1&1&1 \\ 0&1&1&2&2&2&2&2&2&2 \\ 0&1&1&2&2&2&2&2&2&2 \\ 0&1&1&2&2&2&2&2&2&2 \\ 0&1&1&2&2&2&2&3&2&2 \\ 1&1&1&2&2&2&3&2&3&2 \\ 0&1&1&2&2&2&2&3&2&3 \\ 0&1&1&2&2&2&2&2&3&3 \end{smallmatrix} , \ \begin{smallmatrix} 0&0&0&0&0&0&0&0&1&0 \\ 0&1&0&1&1&1&1&1&1&1 \\ 0&0&1&1&1&1&1&1&1&1 \\ 0&1&1&2&2&2&2&2&2&2 \\ 0&1&1&2&2&2&2&2&2&2 \\ 0&1&1&2&2&2&2&2&2&2 \\ 0&1&1&2&2&2&2&2&3&2 \\ 0&1&1&2&2&2&2&3&2&3 \\ 1&1&1&2&2&2&3&2&2&3 \\ 0&1&1&2&2&2&2&3&3&2 \end{smallmatrix} , \ \begin{smallmatrix} 0&0&0&0&0&0&0&0&0&1 \\ 0&1&0&1&1&1&1&1&1&1 \\ 0&0&1&1&1&1&1&1&1&1 \\ 0&1&1&2&2&2&2&2&2&2 \\ 0&1&1&2&2&2&2&2&2&2 \\ 0&1&1&2&2&2&2&2&2&2 \\ 0&1&1&2&2&2&2&2&2&3 \\ 0&1&1&2&2&2&2&2&3&3 \\ 0&1&1&2&2&2&2&3&3&2 \\ 1&1&1&2&2&2&3&3&2&2 \end{smallmatrix} \ 
$$

\begin{theorem} \label{thm:q=1(4)} 
If $q \equiv 1 \mod 4$, the interpolated fusion ring $\mathcal{R}_q$ has rank $(q+5)/2$, $\FPdim \ q(q^2-1)/2$, type $$[[1,1],[(q+1)/2,2],[q-1,(q-1)/4],[q,1],[q+1,(q-5)/4]],$$ and fusion rules
\begin{align*}
x_{\frac{q+1}{2},c_1} x_{\frac{q+1}{2},c_2} & = \delta_{c_1,c_2}x_{1,1} + \delta_{c_1,c_2} \sum_{c_3} \delta_{c_1,c_3}x_{\frac{q+1}{2},c_3} + (1-\delta_{c_1,c_2})\sum_{c_3} x_{q-1,c_3} + x_{q,1} + \delta_{c_1,c_2} \sum_{c_3} x_{q+1,c_3}, \\
x_{\frac{q+1}{2},c_1} x_{q-1,c_2} & = \sum_{c_3} (1-\delta_{c_1,c_3}) x_{\frac{q+1}{2},c_3} + \sum_{c_3}x_{q-1,c_3} +  x_{q,1} + \sum_{c_3} x_{q+1,c_3}, \\
x_{\frac{q+1}{2},c_1} x_{q,1} & = \sum_{c_2} x_{\frac{q+1}{2},c_2} + \sum_{c_2}x_{q-1,c_2} + x_{q,1} + \sum_{c_2} x_{q+1,c_2}, \\
x_{\frac{q+1}{2},c_1} x_{q+1,c_2} & = \sum_{c_3} \delta_{c_1,c_3} x_{\frac{q+1}{2},c_3} + \sum_{c_3}x_{q-1,c_3} +  x_{q,1} +  \sum_{c_3}(1+ \delta_{c_2+c_3,\frac{q-1}{4}})x_{q+1,c_3}, \\
x_{q-1,c_1} x_{q-1,c_2} & = \delta_{c_1,c_2}x_{1,1} + \sum_{c_3} x_{\frac{q+1}{2},c_3}  +  \hspace{-.7cm} \sum_{\substack{c_3 \text{ such that } \\ c_1+c_2+c_3 = \frac{q+1}{2} \\ \text{ or } 2\Max(c_1,c_2,c_3)}} \hspace{-.8cm} x_{q-1,c_3} +  2 \hspace{-.9cm}  \sum_{\substack{c_3 \text{ such that } \\ c_1+c_2+c_3 \neq \frac{q+1}{2} \\ \text{ and } 2\Max(c_1,c_2,c_3)}} \hspace{-1cm} x_{q-1,c_3} + (2-\delta_{c_1,c_2})x_{q,1} +   2\sum_{c_3} x_{q+1,c_3}, \\
x_{q-1,c_1} x_{q,1} & = \sum_{c_2} x_{\frac{q+1}{2},c_2} + \sum_{c_2} (2-\delta_{c_1,c_2})x_{q-1,c_2} + 2x_{q,1} + 2 \sum_{c_2}x_{q+1,c_2}, \\
x_{q-1,c_1} x_{q+1,c_2} & = \sum_{c_3} x_{\frac{q+1}{2},c_3} + 2\sum_{c_3} x_{q-1,c_3} + 2x_{q,1} + 2 \sum_{c_3}x_{q+1,c_3}, \\
x_{q,1} x_{q,1} & = x_{1,1} +  \sum_{c} x_{\frac{q+1}{2},c} + 2\sum_{c} x_{q-1,c} + 2 x_{q,1} + 2 \sum_{c}x_{q+1,c}, \\
x_{q,1} x_{q+1,c_1} & = \sum_{c_2} x_{\frac{q+1}{2},c_2} + 2\sum_{c_2} x_{q-1,c_2} + 2 x_{q,1} + \sum_{c_2}(2+\delta_{c_1,c_2}) x_{q+1,c_2}, \\
x_{q+1,c_1} x_{q+1,c_2} & = \delta_{c_1,c_2}x_{1,1} + (1+ \delta_{c_1+c_2,\frac{q-1}{4}})\sum_{c_3}x_{\frac{q+1}{2},c_3} + 2\sum_{c_3} x_{q-1,c_3} + (2+\delta_{c_1,c_2})x_{q,1} + 3 \hspace{-.9cm} \sum_{\substack{c_3 \text{ such that } \\ c_1+c_2+c_3 = \frac{q-1}{2} \\ \text{ or } 2\Max(c_1,c_2,c_3)}} \hspace{-.8cm} x_{q+1,c_3} +  2 \hspace{-.9cm}  \sum_{\substack{c_3 \text{ such that } \\ c_1+c_2+c_3 \neq \frac{q-1}{2} \\ \text{ and } 2\Max(c_1,c_2,c_3)}} \hspace{-1cm} x_{q+1,c_3}. 
\end{align*}
\end{theorem}
\begin{proof}
The rules are given by Table \ref{table:q=1(4)} and Lemmas \ref{lem:first3} to \ref{lem:last3}.
\end{proof}

The smallest non prime-power $q \equiv 1 \mod 4$ is $q=21$, and $\mathcal{R}_{21}$ has rank $13$, $\FPdim \ 4620$, type $$[[1,1],[11,2],[20,5],[21,1],[22,4]],$$ and fusion matrices:
$$
\begin{smallmatrix} 1&0&0&0&0&0&0&0&0&0&0&0&0 \\ 0&1&0&0&0&0&0&0&0&0&0&0&0 \\ 0&0&1&0&0&0&0&0&0&0&0&0&0 \\ 0&0&0&1&0&0&0&0&0&0&0&0&0 \\ 0&0&0&0&1&0&0&0&0&0&0&0&0 \\ 0&0&0&0&0&1&0&0&0&0&0&0&0 \\ 0&0&0&0&0&0&1&0&0&0&0&0&0 \\ 0&0&0&0&0&0&0&1&0&0&0&0&0 \\ 0&0&0&0&0&0&0&0&1&0&0&0&0 \\ 0&0&0&0&0&0&0&0&0&1&0&0&0 \\ 0&0&0&0&0&0&0&0&0&0&1&0&0 \\ 0&0&0&0&0&0&0&0&0&0&0&1&0 \\ 0&0&0&0&0&0&0&0&0&0&0&0&1 \end{smallmatrix} , \ \begin{smallmatrix} 0&1&0&0&0&0&0&0&0&0&0&0&0 \\ 1&1&0&0&0&0&0&0&1&1&1&1&1 \\ 0&0&0&1&1&1&1&1&1&0&0&0&0 \\ 0&0&1&1&1&1&1&1&1&1&1&1&1 \\ 0&0&1&1&1&1&1&1&1&1&1&1&1 \\ 0&0&1&1&1&1&1&1&1&1&1&1&1 \\ 0&0&1&1&1&1&1&1&1&1&1&1&1 \\ 0&0&1&1&1&1&1&1&1&1&1&1&1 \\ 0&1&1&1&1&1&1&1&1&1&1&1&1 \\ 0&1&0&1&1&1&1&1&1&1&1&1&2 \\ 0&1&0&1&1&1&1&1&1&1&1&2&1 \\ 0&1&0&1&1&1&1&1&1&1&2&1&1 \\ 0&1&0&1&1&1&1&1&1&2&1&1&1 \end{smallmatrix} , \ \begin{smallmatrix} 0&0&1&0&0&0&0&0&0&0&0&0&0 \\ 0&0&0&1&1&1&1&1&1&0&0&0&0 \\ 1&0&1&0&0&0&0&0&1&1&1&1&1 \\ 0&1&0&1&1&1&1&1&1&1&1&1&1 \\ 0&1&0&1&1&1&1&1&1&1&1&1&1 \\ 0&1&0&1&1&1&1&1&1&1&1&1&1 \\ 0&1&0&1&1&1&1&1&1&1&1&1&1 \\ 0&1&0&1&1&1&1&1&1&1&1&1&1 \\ 0&1&1&1&1&1&1&1&1&1&1&1&1 \\ 0&0&1&1&1&1&1&1&1&1&1&1&2 \\ 0&0&1&1&1&1&1&1&1&1&1&2&1 \\ 0&0&1&1&1&1&1&1&1&1&2&1&1 \\ 0&0&1&1&1&1&1&1&1&2&1&1&1 \end{smallmatrix} , \ \begin{smallmatrix} 0&0&0&1&0&0&0&0&0&0&0&0&0 \\ 0&0&1&1&1&1&1&1&1&1&1&1&1 \\ 0&1&0&1&1&1&1&1&1&1&1&1&1 \\ 1&1&1&2&1&2&2&2&1&2&2&2&2 \\ 0&1&1&1&2&1&2&2&2&2&2&2&2 \\ 0&1&1&2&1&2&1&2&2&2&2&2&2 \\ 0&1&1&2&2&1&2&1&2&2&2&2&2 \\ 0&1&1&2&2&2&1&1&2&2&2&2&2 \\ 0&1&1&1&2&2&2&2&2&2&2&2&2 \\ 0&1&1&2&2&2&2&2&2&2&2&2&2 \\ 0&1&1&2&2&2&2&2&2&2&2&2&2 \\ 0&1&1&2&2&2&2&2&2&2&2&2&2 \\ 0&1&1&2&2&2&2&2&2&2&2&2&2 \end{smallmatrix} , \ \begin{smallmatrix} 0&0&0&0&1&0&0&0&0&0&0&0&0 \\ 0&0&1&1&1&1&1&1&1&1&1&1&1 \\ 0&1&0&1&1&1&1&1&1&1&1&1&1 \\ 0&1&1&1&2&1&2&2&2&2&2&2&2 \\ 1&1&1&2&2&2&1&2&1&2&2&2&2 \\ 0&1&1&1&2&2&2&1&2&2&2&2&2 \\ 0&1&1&2&1&2&2&1&2&2&2&2&2 \\ 0&1&1&2&2&1&1&2&2&2&2&2&2 \\ 0&1&1&2&1&2&2&2&2&2&2&2&2 \\ 0&1&1&2&2&2&2&2&2&2&2&2&2 \\ 0&1&1&2&2&2&2&2&2&2&2&2&2 \\ 0&1&1&2&2&2&2&2&2&2&2&2&2 \\ 0&1&1&2&2&2&2&2&2&2&2&2&2 \end{smallmatrix},$$
\vspace*{.1cm}
$$\begin{smallmatrix} 0&0&0&0&0&1&0&0&0&0&0&0&0 \\ 0&0&1&1&1&1&1&1&1&1&1&1&1 \\ 0&1&0&1&1&1&1&1&1&1&1&1&1 \\ 0&1&1&2&1&2&1&2&2&2&2&2&2 \\ 0&1&1&1&2&2&2&1&2&2&2&2&2 \\ 1&1&1&2&2&2&2&1&1&2&2&2&2 \\ 0&1&1&1&2&2&1&2&2&2&2&2&2 \\ 0&1&1&2&1&1&2&2&2&2&2&2&2 \\ 0&1&1&2&2&1&2&2&2&2&2&2&2 \\ 0&1&1&2&2&2&2&2&2&2&2&2&2 \\ 0&1&1&2&2&2&2&2&2&2&2&2&2 \\ 0&1&1&2&2&2&2&2&2&2&2&2&2 \\ 0&1&1&2&2&2&2&2&2&2&2&2&2 \end{smallmatrix} , \ \begin{smallmatrix} 0&0&0&0&0&0&1&0&0&0&0&0&0 \\ 0&0&1&1&1&1&1&1&1&1&1&1&1 \\ 0&1&0&1&1&1&1&1&1&1&1&1&1 \\ 0&1&1&2&2&1&2&1&2&2&2&2&2 \\ 0&1&1&2&1&2&2&1&2&2&2&2&2 \\ 0&1&1&1&2&2&1&2&2&2&2&2&2 \\ 1&1&1&2&2&1&2&2&1&2&2&2&2 \\ 0&1&1&1&1&2&2&2&2&2&2&2&2 \\ 0&1&1&2&2&2&1&2&2&2&2&2&2 \\ 0&1&1&2&2&2&2&2&2&2&2&2&2 \\ 0&1&1&2&2&2&2&2&2&2&2&2&2 \\ 0&1&1&2&2&2&2&2&2&2&2&2&2 \\ 0&1&1&2&2&2&2&2&2&2&2&2&2 \end{smallmatrix} , \ \begin{smallmatrix} 0&0&0&0&0&0&0&1&0&0&0&0&0 \\ 0&0&1&1&1&1&1&1&1&1&1&1&1 \\ 0&1&0&1&1&1&1&1&1&1&1&1&1 \\ 0&1&1&2&2&2&1&1&2&2&2&2&2 \\ 0&1&1&2&2&1&1&2&2&2&2&2&2 \\ 0&1&1&2&1&1&2&2&2&2&2&2&2 \\ 0&1&1&1&1&2&2&2&2&2&2&2&2 \\ 1&1&1&1&2&2&2&2&1&2&2&2&2 \\ 0&1&1&2&2&2&2&1&2&2&2&2&2 \\ 0&1&1&2&2&2&2&2&2&2&2&2&2 \\ 0&1&1&2&2&2&2&2&2&2&2&2&2 \\ 0&1&1&2&2&2&2&2&2&2&2&2&2 \\ 0&1&1&2&2&2&2&2&2&2&2&2&2 \end{smallmatrix} , \ \begin{smallmatrix} 0&0&0&0&0&0&0&0&1&0&0&0&0 \\ 0&1&1&1&1&1&1&1&1&1&1&1&1 \\ 0&1&1&1&1&1&1&1&1&1&1&1&1 \\ 0&1&1&1&2&2&2&2&2&2&2&2&2 \\ 0&1&1&2&1&2&2&2&2&2&2&2&2 \\ 0&1&1&2&2&1&2&2&2&2&2&2&2 \\ 0&1&1&2&2&2&1&2&2&2&2&2&2 \\ 0&1&1&2&2&2&2&1&2&2&2&2&2 \\ 1&1&1&2&2&2&2&2&2&2&2&2&2 \\ 0&1&1&2&2&2&2&2&2&3&2&2&2 \\ 0&1&1&2&2&2&2&2&2&2&3&2&2 \\ 0&1&1&2&2&2&2&2&2&2&2&3&2 \\ 0&1&1&2&2&2&2&2&2&2&2&2&3 \end{smallmatrix} , \ \begin{smallmatrix} 0&0&0&0&0&0&0&0&0&1&0&0&0 \\ 0&1&0&1&1&1&1&1&1&1&1&1&2 \\ 0&0&1&1&1&1&1&1&1&1&1&1&2 \\ 0&1&1&2&2&2&2&2&2&2&2&2&2 \\ 0&1&1&2&2&2&2&2&2&2&2&2&2 \\ 0&1&1&2&2&2&2&2&2&2&2&2&2 \\ 0&1&1&2&2&2&2&2&2&2&2&2&2 \\ 0&1&1&2&2&2&2&2&2&2&2&2&2 \\ 0&1&1&2&2&2&2&2&2&3&2&2&2 \\ 1&1&1&2&2&2&2&2&3&2&3&2&2 \\ 0&1&1&2&2&2&2&2&2&3&2&3&2 \\ 0&1&1&2&2&2&2&2&2&2&3&2&3 \\ 0&2&2&2&2&2&2&2&2&2&2&3&2 \end{smallmatrix},$$
\vspace*{.1cm}
$$\begin{smallmatrix} 0&0&0&0&0&0&0&0&0&0&1&0&0 \\ 0&1&0&1&1&1&1&1&1&1&1&2&1 \\ 0&0&1&1&1&1&1&1&1&1&1&2&1 \\ 0&1&1&2&2&2&2&2&2&2&2&2&2 \\ 0&1&1&2&2&2&2&2&2&2&2&2&2 \\ 0&1&1&2&2&2&2&2&2&2&2&2&2 \\ 0&1&1&2&2&2&2&2&2&2&2&2&2 \\ 0&1&1&2&2&2&2&2&2&2&2&2&2 \\ 0&1&1&2&2&2&2&2&2&2&3&2&2 \\ 0&1&1&2&2&2&2&2&2&3&2&3&2 \\ 1&1&1&2&2&2&2&2&3&2&2&2&3 \\ 0&2&2&2&2&2&2&2&2&3&2&2&2 \\ 0&1&1&2&2&2&2&2&2&2&3&2&3 \end{smallmatrix} , \ \begin{smallmatrix} 0&0&0&0&0&0&0&0&0&0&0&1&0 \\ 0&1&0&1&1&1&1&1&1&1&2&1&1 \\ 0&0&1&1&1&1&1&1&1&1&2&1&1 \\ 0&1&1&2&2&2&2&2&2&2&2&2&2 \\ 0&1&1&2&2&2&2&2&2&2&2&2&2 \\ 0&1&1&2&2&2&2&2&2&2&2&2&2 \\ 0&1&1&2&2&2&2&2&2&2&2&2&2 \\ 0&1&1&2&2&2&2&2&2&2&2&2&2 \\ 0&1&1&2&2&2&2&2&2&2&2&3&2 \\ 0&1&1&2&2&2&2&2&2&2&3&2&3 \\ 0&2&2&2&2&2&2&2&2&3&2&2&2 \\ 1&1&1&2&2&2&2&2&3&2&2&2&3 \\ 0&1&1&2&2&2&2&2&2&3&2&3&2 \end{smallmatrix} , \ \begin{smallmatrix} 0&0&0&0&0&0&0&0&0&0&0&0&1 \\ 0&1&0&1&1&1&1&1&1&2&1&1&1 \\ 0&0&1&1&1&1&1&1&1&2&1&1&1 \\ 0&1&1&2&2&2&2&2&2&2&2&2&2 \\ 0&1&1&2&2&2&2&2&2&2&2&2&2 \\ 0&1&1&2&2&2&2&2&2&2&2&2&2 \\ 0&1&1&2&2&2&2&2&2&2&2&2&2 \\ 0&1&1&2&2&2&2&2&2&2&2&2&2 \\ 0&1&1&2&2&2&2&2&2&2&2&2&3 \\ 0&2&2&2&2&2&2&2&2&2&2&3&2 \\ 0&1&1&2&2&2&2&2&2&2&3&2&3 \\ 0&1&1&2&2&2&2&2&2&3&2&3&2 \\ 1&1&1&2&2&2&2&2&3&2&3&2&2 \end{smallmatrix} \
$$

\section{General results on simple fusion categories} \label{sec:simple}

This section proves general results on simple fusion categories (always assumed over $\mathbb{C}$). It proves that a non-pointed simple fusion category is non-braided if and only if its Drinfeld center is simple. In particular, a simple fusion category Morita equivalent to a simple fusion category with a noncommutative Grothendieck ring, is non-braided and has a simple Drinfeld center. This section also provides a nice reformulation of the weakly group-theoretical conjecture in the simple case, namely, it proves that every simple integral fusion category is weakly group-theoretical if and only if every simple integral modular fusion category is pointed. We checked by computer assistance that every simple integral modular fusion category of rank up to $12$ is pointed, see \cite{ABPP}. We refer to \cite{EGNO15} for the notions of degenerate, braided, symmetric, modular fusion categories, M\"uger centralizer and Drinfeld center.  The beginning of the proof of \cite[Theorem 9.12]{ENO11} leads us to the following result. Its detailed proof is due to Victor Ostrik.
\begin{theorem} \label{simpleZ}
Let $\mathcal{C}$ be a non-pointed simple fusion category over $\mathbb{C}$. Then the Drinfeld center $\mathcal{Z}(\mathcal{C})$ is simple if and only if $\mathcal{C}$ is non-braided.
\end{theorem}
\begin{proof}
One way is well-known: if $\mathcal{C}$ is braided (without further assumption) then $\mathcal{Z}(\mathcal{C})$ has a subcategory equivalent to $\mathcal{C}$, so is not simple.
Assume now that $\mathcal{Z}(\mathcal{C})$ is not simple, so it contains a proper non-trivial subcategory $\mathcal{B}$. The image of $\mathcal{B}$ in $\mathcal{C}$ under the forgetful functor $F:\mathcal{Z}(\mathcal{C}) \to \mathcal{C}$ is some subcategory $\mathcal{D}$ of $\mathcal{C}$. Since $\mathcal{C}$ is simple, then $\mathcal{D}$ is $\VVec$ or $\mathcal{C}$. 

If $\mathcal{D} = \VVec$ then the functor $F_{|B}: B \to Vec$ is central as a restriction of the central functor $F$, so it coincides with the braided tensor functor $\mathcal{B} \to \mathcal{Z}(\VVec) = \VVec$. Then $B$ is symmetric (because $F$ is faithful) and Tannakian by \cite[Definition 9.9.16]{EGNO15}. So $\mathcal{C}$ has a non-trivial grading by \cite[Proposition 2.9]{ENO11}. Since $\mathcal{C}$ is simple, it must be pointed, contradiction.
Next, if $\mathcal{D} = \mathcal{C}$, then the functor $\mathcal{B} \to \mathcal{C}$ is surjective. So by \cite[Theorem 3.12]{DNO}, the functor $\mathcal{B}' \to \mathcal{C}$ is injective, where $\mathcal{B}'$ is the M\"uger centralizer of $\mathcal{B}$. But it is also surjective as above. Hence the functor $\mathcal{B}' \to \mathcal{C}$ is both injective and surjective, so is an equivalence. Thus $\mathcal{C}$ is equivalent to something braided, so is braided itself.
\end{proof}  

The following result should be well-known to experts: 

\begin{theorem} \label{thm:braiding}
A simple spherical braided fusion category $\mathcal{C}$ (over $\mathbb{C}$) must be modular or equivalent (as braided fusion category) to $\Rep(G)$, for some finite simple group $G$.
\end{theorem}
\begin{proof}
Because $\mathcal{C}$ is braided and spherical, it is also pre-modular by \cite[Definition 8.13.1]{EGNO15}. If it is non-degenerate (i.e.  M\"uger centralizer $\mathcal{C}' = \VVec$, by \cite[Theorem 8.20.7]{EGNO15}) then it is modular by \cite[Definition 8.13.4 and Remark 8.19.3]{EGNO15}.  If it is degenerate (i.e. $\mathcal{C}' \neq \VVec $) then by simplicity $\mathcal{C}'=\mathcal{C}$, so $\mathcal{C}$ is symmetric, and by \cite{Del} and simplicity, $\mathcal{C} \simeq \Rep(G)$ as braided fusion category (see \cite[Example 2,1]{Na}), with $G$ a finite simple group.
\end{proof}

\begin{proposition} \label{prop:IntSph}
A weakly-integral fusion category $\mathcal{C}$ (over $\mathbb{C}$) is spherical.
\end{proposition}
\begin{proof}
Because $\mathcal{C}$ is weakly-integral, it is also pseudo-unitary by \cite[Proposition 9.6.5]{EGNO15}, and so spherical by \cite[Proposition 9.5.1]{EGNO15}.
\end{proof}

\begin{corollary}  \label{cor:braiding}
A simple weakly-integral braided fusion category $\mathcal{C}$ (over $\mathbb{C}$) must be modular or equivalent (as braided fusion category) to $\Rep(G)$, for some finite simple group $G$.
\end{corollary}
\begin{proof}
Immediate from Proposition \ref{prop:IntSph} and Theorem \ref{thm:braiding}.
\end{proof}

The weakly group-theoretical conjecture (supporting a negative answer to \cite[Question 2]{ENO11}) states as follows:
\begin{statement}[Open] \label{sta:wgtconj}
Every integral fusion category is weakly group-theoretical.
\end{statement}

We wonder whether Statement \ref{sta:wgtconj} can be deduced from its following simple case.

\begin{statement}[Open] \label{sta:wgtconjsimple}
Every simple integral fusion category is weakly group-theoretical.
\end{statement}

The next theorem shows that Statement \ref{sta:wgtconjsimple} can be reformulated in the following nice way:

\begin{statement}[Open] \label{sta:wgtconjsimplemodular}
Every simple integral modular fusion category is pointed.
\end{statement}

\begin{theorem} \label{thm:S2S3}
Statement \ref{sta:wgtconjsimple} is equivalent to Statement \ref{sta:wgtconjsimplemodular}.
\end{theorem}
\begin{proof}
First, let us prove that Statement \ref{sta:wgtconjsimple} implies Statement \ref{sta:wgtconjsimplemodular}: let $\mathcal{C}$ be a simple integral modular fusion category, then by Statement \ref{sta:wgtconjsimple}, it is weakly group-theoretical, so by Theorem \ref{thm:wgtsimple}, if it is non-pointed then it is $\Rep(G)$ with $G$ a non-abelian finite simple group, contradiction with the modular assumption (see \cite[Example 2.1]{Na}).

Next, let us prove that Statement \ref{sta:wgtconjsimplemodular} implies Statement \ref{sta:wgtconjsimple}: let $\mathcal{C}$ be a simple integral fusion category. If it is pointed then it is weakly group-theoretical, so we can assume that it is non-pointed, and so by Statement \ref{sta:wgtconjsimplemodular}, it is not modular. If it is still braided, then by Corollary \ref{cor:braiding} it must be equivalent to $\Rep(G)$, so weakly group-theoretical. Thus we can assume that it is non-braided, and then by Theorem \ref{simpleZ}, $\mathcal{Z}(\mathcal{C})$ must be simple, but it is also integral and modular, so by Statement \ref{sta:wgtconjsimplemodular}, it must be pointed, contradiction with $\mathcal{C}$ non-pointed.
\end{proof}
Let us state here three theorems used in above proof or below remark.  

\begin{theorem}[Proposition 8.14.6 in \cite{EGNO15}] \label{thm:modalg}
Let $\mathcal{C}$ be a modular category and let $X$ be a simple object. Then $\dim(\mathcal{C})/\dim(X)^2$ is an algebraic integer.
\end{theorem}

\begin{theorem}[Corollary 6.16 in \cite{nik}] \label{thm:PrimePower}
Let $\mathcal{C}$ be a non-degenerate integral braided complex fusion category. If there is a simple object of prime-power $\FPdim$ then there is a nontrivial symmetric subcategory.
\end{theorem}


\begin{theorem}[Proposition 9.11 in \cite{ENO11}] \label{thm:wgtsimple}
A non-pointed weakly group-theoretical simple fusion category is equal to $\Rep(G)$ with $G$ a non-abelian finite simple group.
\end{theorem} 

\begin{remark} By combining Theorems \ref{thm:modalg} and \ref{thm:PrimePower} together with a quick brute-force computation involving Egyptian fractions with squared denominators, we can show that a simple integral modular fusion category of rank up to $12$ is pointed,, see \cite{ABPP}. Then by Corollary \ref{cor:braiding}, a simple integral braided fusion category of rank up to $12$ must be $\Rep(G)$, for some finite simple group $G$.
\end{remark}

Here are some other nice consequences on simple fusion categories.

\begin{corollary} \label{cor:SimNC}
Let $\mathcal{C}$ be a simple fusion category with a noncommutative Grothendieck ring. Then the Drinfeld center $\mathcal{Z}(\mathcal{C})$ is simple.
\end{corollary}
\begin{proof}
Immediate from Theorem \ref{simpleZ} as $\mathcal{C}$ is obviously non-pointed and non-braided.
\end{proof}
For example, the Extended Haagerup categories \cite{ExtHaag} $\mathcal{EH}_i$, $i \in \{2,3,4\}$ are simple with a noncommutative Grothendieck ring, so $\mathcal{Z}(\mathcal{EH}_i)$ is simple.

\begin{corollary}
Let $\mathcal{C}$ be a simple fusion category which is Morita equivalent to a simple fusion category with a noncommutative Grothendieck ring. Then it is non-braided.
\end{corollary}
\begin{proof}
Immediate from Theorem \ref{simpleZ} and Corollary \ref{cor:SimNC}, as two Morita equivalent fusion categories have the same Drinfeld center (up to equivalence), and such $\mathcal{C}$ cannot be pointed.
\end{proof}
For example, the Extended-Haagerup fusion category $\mathcal{EH}_1$, the one with a commutative Grothendieck ring, is Morita equivalent to $\mathcal{EH}_i$ by \cite{ExtHaag}, so it is non-braided. It is a third proof of this result after the one using Morrison-Walker induction matrix \cite{MoWa}, and the one using \cite[Proposition 5.2]{ENO21}, i.e. a spherical braided fusion category is \emph{Isaacs} (see \S\ref{sub:isaacs}), together with \cite[Proposition 5.1]{BP}, i.e. $\mathcal{EH}_1$ is not Isaacs.
Let us finish this section by two remarks about the weakly group-theoretical conjecture.

\begin{remark} By \cite[Proposition 2.9]{ENO11} and Theorem \ref{thm:PrimePower}, to exclude a non group-like simple integral fusion ring from complex categorification, it is sufficient to show that its Drinfeld center would have a non-trivial simple object of prime-power dimension, which can be checked by computing the induction matrix. The paper \cite{MoWa} provides an algorithm to do so, unfortunately its complexity is too big on the simplest example of our interpolated family (i.e. $\mathcal{R}_6$).
\end{remark}


\begin{remark}
By Theorem \ref{thm:S2S3}, a way to prove the weakly group-theoretical conjecture would be to first prove that it reduces to its simple case, and next to deduce a contradiction from the existence of a non-pointed simple integral modular fusion category.
\end{remark}

\section{No braided complex categorification} \label{sec:braid}
The interpolated fusion rings $(\mathcal{R}_q)$ pass all the known categorification criteria (checked in \S\ref{sec:criteria}). This leads to wonder whether some ones can be categorified (with $q$ non prime-power). We could naively expect that the F-symbols can be interpolated too, which is not clear, at least without modification. What is true is that the symmetric structure cannot be interpolated, so something else could be broken somewhere. What about the braiding structure? In this section we prove that a complex categorification (if any) of an interpolated fusion ring $\mathcal{R}_q$ (with $q$ non prime-power) cannot be braided, and so its Drinfeld center must be simple. The following result is partially due to an anonymous referee.

\begin{corollary} \label{cor:nobraiding}
For $q$ non prime-power, the interpolated fusion ring $\mathcal{R}_q$ admits no braided complex categorification.
\begin{proof}
Assume that it admits a braided complex categorification $\mathcal{C}$, then we can apply Corollary \ref{cor:braiding}, and by Lemma \ref{lem:NoOverlap}, $\mathcal{C}$ must be modular. Thus by Theorem \ref{thm:modalg}, pseudo-unitarity and integrality, for every simple object $x$ then $\FPdim(x)^2$ divides $\FPdim(\mathcal{C})$; but $\FPdim(x_{q,1}) = q$ whereas $\FPdim(\mathcal{C}) = q(q^2-1)$ or $q(q^2-1)/2$, contradiction.
\end{proof}
\end{corollary}

\begin{remark} Assume the existence of $q$ non prime-power such that $\mathcal{R}_q$ admits a complex categorification $\mathcal{C}_q$. Then by Corollary \ref{cor:nobraiding} and Theorem \ref{simpleZ}, the Drinfeld center $\mathcal{Z}(\mathcal{C}_q)$ is simple. Contrariwise, if $q$ is a prime-power then $\Rep(\PSL(2,q))$ is braided so its Drinfeld center is not simple.
\end{remark}

\section{Categorification criteria checking} \label{sec:criteria}
In this section we check seven categorification criteria on the interpolated fusion rings $\mathcal{R}_q$.
\subsection{Schur product criterion}
Let $\mathcal{F}$ be a commutative fusion ring with eigentable $(\lambda_{i,j})$ and $\lambda_{i,1} = \max_j(|\lambda_{i,j}|)$. Here is the commutative Schur product criterion: 
\begin{theorem}[Corollary 8.5 in \cite{LPW20}] \label{thm:schur}  
If $\mathcal{F}$ admits a unitary categorification then for all triple $(j_1,j_2,j_3)$ we have $$\sum_i \frac{\lambda_{i,j_1}\lambda_{i,j_2}\lambda_{i,j_3}}{\lambda_{i,1}} \ge 0.$$
\end{theorem}

Note that Theorem \ref{thm:schur} is the corollary of a (less tractable) noncommutative version \cite[Proposition 8.3]{LPW20}, which is an application of the quantum Schur product theorem on subfactors \cite[Theorem 4.1]{Liuex}. See also a new proof in \cite{ENO21}.

\begin{lemma} \label{lem:supremum}
If $\sum_{i=0}^n a_ix^i = 0$ and $|a_n| \ge 1$ then $|x| < 1 + \max_i |a_i|$.
\end{lemma}
\begin{proof}
We can assume $|x|>1$ (otherwise we are done). Let $m:=\max_i |a_i|$. We have $a_n x^n =  - \sum_{i=0}^{n-1} a_i x^i$, so $|a_n x^n| = |\sum_{i=0}^{n-1} a_i x^i|$ and then by triangle inequality
$$|a_n| |x|^n \le \sum_{i=0}^{n-1} |a_i||x|^i \le m \sum_{i=0}^{n-1} |x|^i = m \frac{|x|^n - 1}{|x|-1} < m \frac{|x|^n}{|x|-1},$$
but $|a_n| \ge 1$, thus
$$|x|^n  < \frac{m}{|a_n|}  \frac{|x|^n}{|x|-1} \le m \frac{|x|^n}{|x|-1},$$
It follows that $1  < \frac{m}{|x|-1}$, so  $|x|-1 < m$ and  $|x|<m+1$.
\end{proof}

\begin{proposition} All the fusion rings of the interpolated family ($\mathcal{R}_q$) pass Schur product criterion.
\end{proposition}
\begin{proof}
According to the calculations in \S\ref{sec:cal} and the fact that the variables $k$ and $c$ play the same role in the tables of \S\ref{sec:tab}, it is clear that we are reduced to consider rational functions and to show that they are non-negative. Let first compute one example of such rational funtion, say for $q$ even and $j_1,j_2,j_3$ given by the last column of Table \ref{table:q=0(2)} (then we will see a general argument). As for the proof of Lemma \ref{lem:fullcomput}: 
\begin{align*}
1 + \frac{1}{q-1}\sum_{c=1}^{q/2} \prod_{i=1}^3(-\zeta_{q+1}^{k_ic} - \zeta_{q+1}^{-k_ic}) - \frac{1}{q} & = 1-\frac{1}{q} - \frac{1}{q+1} \frac{1}{2} \sum_{\epsilon_i=\pm 1} \sum_{c=1}^{q}  (\zeta_{q+1}^{\sum_{i=1}^3 \epsilon_i k_i})^c  \\
& =
 \left\{
    \begin{array}{lll}
        1 - \frac{1}{q} - \frac{1}{q+1} \frac{1}{2}(-1) = \frac{2q^2+q-2}{2q(q+1)}& \mbox{if}  & \gcd(q+1, \sum_{i=1}^3 \epsilon_i k_i)<q+1 \\
        1 - \frac{1}{q} - \frac{1}{q+1} \frac{1}{2}(q) = \frac{q^2-2}{2q(q+1)} & \mbox{else}&
    \end{array}
\right.
\end{align*} 

In general, we already know that these rational functions are non-negative for $q$ prime-power (because if $G$ is a finite group then $\Rep(G)$ is a unitary fusion category), so they must be non-negative for $q$ large enough. It remains to show that they are also non-negative for $q$ small.  Let $P/Q$ be such a rational function, and assume the existence of $q_0$ small such that $(P/Q)(q_0)<0$. We know that for $q$ large enough $(P/Q)(q)>0$, so it must exist $x_0$ such that  $q_0 < x_0 < q$ and $P(x_0)=0$. Now, it is clear according to the calculations in \S\ref{sec:cal} that the polynomial $P$ can be chosen with integer coefficients in the interval $[-9,9]$. So $x_0 < 10$ by Lemma \ref{lem:supremum}, but the only positive integer $q < 10$ which is not a prime-power is $q=6$, and $\mathcal{R}_6$ was checked in \cite{LPW20} as $\mathcal{F}_{210}$. The result follows.
\end{proof}

\subsection{Ostrik criterion}
Here is the commutative version of a criterion by V. Ostrik, using the formal codegrees $(\mathfrak{c}_i)$.

\begin{theorem}[Theorem 2.21 in \cite{Ost15}]
Let $\mathcal{F}$ be a commutative fusion ring. If it admits a pseudo-unitary complex categorification, then $2\sum_j \frac{1}{\mathfrak{c}_j^2} \le 1+\frac{1}{\mathfrak{c}_1}$.
\end{theorem} 

\begin{proposition} All the fusion rings of the interpolated family ($\mathcal{R}_q$) pass Ostrik criterion.
\end{proposition}
\begin{proof} The formal codegrees can be computed from the class sizes in \S\ref{sec:tab}, see Remark \ref{rem:orb-stab} and \ref{rk:interational}.
\begin{align*} \sum_j 1/\mathfrak{c}_j^2 = & \left\{
    \begin{array}{lll}
      \frac{1}{q^2(q^2-1)^2}+\frac{1}{q^2}+\frac{q-2}{2}\frac{1}{(q-1)^2} + \frac{q}{2}\frac{1}{(q+1)^2}=\frac{q^5-q^3-3q^2+2}{q^2(q^2-1)^2}  & \mbox{if} & q \text{ even } \\
 \frac{4}{q^2(q^2-1)^2} + \frac{2}{q^2} + \frac{4}{(q-1)^3}\frac{q-3}{4} + \frac{4}{(q+1)^3}\frac{q-3}{4}+\frac{1}{(q+1)^2} = \frac{2q^5 - 3q^4 - 9q^2 + 6}{q^2(q^2-1)^2}    & \mbox{if} & q \equiv -1 \mod 4 \\
\frac{4}{q^2(q^2-1)^2} + \frac{2}{q^2} + \frac{4}{(q-1)^3}\frac{q-5}{4} + \frac{4}{(q+1)^3}\frac{q-1}{4}+\frac{1}{(q-1)^2} = \frac{2q^5 - 3q^4 - 4q^3 - 9q^2 + 6}{q^2(q^2-1)^2}        & \mbox{if} & q \equiv 1 \mod 4
    \end{array}
\right. \\ < & \frac{1}{2} < \frac{1}{2}(1+\frac{1}{\mathfrak{c}_1}).   \qedhere
\end{align*}
\end{proof}

\subsection{Drinfeld center criterion} \label{sub:drinfeld}
This criterion is also a commutative version of a criterion by V. Ostrik involving the formal codegrees $(\mathfrak{c}_i)$. Recall that an integral fusion category over $\mathbb{C}$ is pseudo-unitary by \cite[Propositions 9.6.5]{EGNO15}.

\begin{theorem} \label{thm:drinfeld}
Let $\mathcal{F}$ be a commutative fusion ring. If it admits a pseudo-unitary complex categorification then $\mathfrak{c}_1/\mathfrak{c}_i$ is an algebraic integer for all $i$.
\end{theorem}
\begin{proof}
By pseudo-unitarity, the categorical dim is $\mathfrak{c}_1$ (the $\FPdim$), so the result follows by \cite[Corollary 2.14]{Ost15}.
\end{proof}

Recall that a pseudo-unitary fusion category over $\mathbb{C}$ is spherical by \cite[Propositions 9.5.1]{EGNO15}. If $\mathcal{F}$ admits a complex spherical categorification $\mathcal{C}$, then $(\mathfrak{c}_1/\mathfrak{c}_i)$ are exactly the $\FPdim$ of the simple objects of the Drinfeld center which contain the trivial object in $\mathcal{C}$ under the forgetful functor, by \cite[Theorem 2.13]{Ost15}.   

\begin{proposition} All the fusion rings of the interpolated family ($\mathcal{R}_q$) pass Drinfeld center criterion.
\end{proposition}
\begin{proof}
The numbers $(\mathfrak{c}_1/\mathfrak{c}_i)$ are precisely the \textit{class sizes} mentioned in the tables of \S\ref{sec:tab}, they are clearly integers, and they interpolate as expected by Remarks \ref{rem:orb-stab} and \ref{rk:interational}.
\end{proof}

\subsection{Extended cyclotomic criterion}
The following theorem is a slight extension of the usual cyclotomic criterion (on the simple object FPdims) of a fusion ring to all the entries of its eigentable, in the commutative case.  

\begin{theorem} \label{thm:cyclo}
Let $\mathcal{F}$ be a commutative fusion ring, and let $(\lambda_{i,j})$ be its eigentable. If $\mathcal{F}$ admits a complex categorification then $\lambda_{i,j}$ is a cyclotomic integer, for all $i,j$.
\end{theorem}
\begin{proof}
It follows from \cite[Theorem 8.51]{ENO05} and \cite[Theorem 2.11]{LPR2}.
\end{proof}

\begin{proposition} All the fusion rings of the interpolated family ($\mathcal{R}_q$) pass the extended cyclotomic criterion.
\end{proposition}
\begin{proof}
All the entries in the tables of \S\ref{sec:tab} are clearly cyclotomic integers (for the ones with a square root, consider the quadratic Gauss sums).
\end{proof}

\subsection{Isaacs criterion} \label{sub:isaacs}
Isaacs' \cite[Theorem 3.7]{Isa} let us wondered about a stronger version of Theorem \ref{thm:drinfeld}.

\begin{definition}[Isaacs property] A commutative fusion ring $\mathcal{F}$ with eigentable $(\lambda_{i,j})$ is called \emph{Isaacs} if $\displaystyle{\frac{\lambda_{i,j}\mathfrak{c}_1}{\lambda_{i,1}\mathfrak{c}_j}}$ is an algebraic integer for all $i,j$.
\end{definition} 

As observed by P. Etingof \cite{eti20} (see the details in \cite{LPR2} or \cite{ENO21}), if a commutative fusion ring is Isaacs then it is of \emph{Frobenius type} (i.e. $\frac{\mathfrak{c}_1}{\lambda_{i,1}}$ is an algebraic integer, for all $i$), which is related to the generalization of Kaplansky's 6th conjecture to complex fusion categories.

\begin{theorem}[Proposition 5.2 in \cite{ENO21}] A complex spherical braided fusion category is \emph{Isaacs}.
\end{theorem}
\noindent It  makes the Isaacs property a criterion for complex spherical braided categorification of commutative fusion rings. Now, we know (\S\ref{sec:braid}) that the interpolated fusion rings $\mathcal{R}_q$ are non-braided for $q$ non prime-power, but they are still Isaacs:

\begin{proposition} All the fusion rings of the interpolated family ($\mathcal{R}_q$) are Isaacs.
\end{proposition}
\begin{proof}
The case $j=1$ or $\lambda_{i,j}=0$ are obvious. If $j \neq 1$ and $\lambda_{i,j}\neq 0$, then observe on the tables that $\frac{\mathfrak{c}_1}{\lambda_{i,1}\mathfrak{c}_j}$ is an integer, whereas $\lambda_{i,j}$ is a cyclotomic integer (by Theorem \ref{thm:cyclo}).
\end{proof}


\subsection{Zero spectrum criterion}
Here is a general categorification criterion over every field. It corresponds to the existence of a pentagon equation of the form $xy=0$, with $x,y \neq 0$ (see examples as $\mathcal{F}_{660}$ in \cite{LPR1}).

\begin{theorem}[see \cite{LPR1}] \label{thm:zerocrit} 
Let $\mathcal{F}$ be a fusion ring with fusion rules $(N_{i,j}^{k})$. If there are indices $i_1, \dots, i_9$, such that
\begin{align}
N_{i_4,i_1}^{i_6}, \ N_{i_5,i_4}^{i_2}, \ N_{i_5,i_6}^{i_3},& \ N_{i_7,i_9}^{i_1}, \ N_{i_2,i_7}^{i_8}, \ N_{i_8,i_9}^{i_3} \neq 0, \label{line:1} \\
\sum_{k} N_{i_4,i_7}^{k} N_{i_5^*,i_8}^{k} N_{i_6,i_9^*}^{k}&=0, \label{line:2} \\  
N_{i_2,i_1}^{i_{3}}&=1, \label{line:3}  \\ 
\sum_{k} N_{i_5,i_4}^{k} N_{i_3,i_1^*}^{k} &=1 
~\text{or}~
\sum_{k} N_{i_2,i_4^*}^{k} N_{i_3,i_6^*}^{k} =1 
~\text{or}~
\sum_{k} N_{i_5^*,i_2}^{k} N_{i_6,i_1^*}^{k} =1, 
\\
\sum_{k} N_{i_2,i_7}^{k} N_{i_3,i_9^*}^{k} &=1 
~\text{or}~
\sum_{k} N_{i_8,i_7^*}^{k} N_{i_3,i_1^*}^{k} =1
~\text{or}~
\sum_{k} N_{i_2^*,i_8}^{k} N_{i_1,i_9^*}^{k} =1.
\end{align}
then $\mathcal{F}$ cannot be categorified at all, i.e. it is not the Grothendieck ring of a fusion category, over any field.
\end{theorem}

Let us take the index $1$ for the trivial object.

\begin{lemma} \label{lem:interzero}
Let $\mathcal{F}$ be a fusion ring, and assume the existence of indices $i_1, \dots, i_9$ such that (\ref{line:1}) and (\ref{line:2}) hold. Then $i_j \neq 1$,  for all $j \in \{4,\dots,9 \}$.
\end{lemma}
\begin{proof}
Assume that there is $j \in \{4,\dots,9 \}$ such that $i_j=1$. We will reach a contradiction. Let write the details for $j=4$ (the other cases are similar): assume that $i_4=1$, by (\ref{line:1}) and the neutral axiom of fusion ring, $i_1=i_6$ and $i_2=i_5$; in addition, $N_{1,i_7}^k = \delta_{i_7,k}$. Then, by (\ref{line:2}), $N_{i_5^*,i_8}^{i_7} N_{i_6,i_9^*}^{i_7}=0$. Thus, using Frobenius reciprocity, $N_{i_2,i_7}^{i_8}=0$ or $N_{i_7,i_9}^{i_1}=0$, contradiction with (\ref{line:1}). 
\end{proof}

\begin{lemma} \label{lem:prezero}
Let $\mathcal{F}$ be a fusion ring such that there is $k_0$ with $N_{i,j}^{k_0} \neq 0$ for all $i,j \neq 1$. Then $\mathcal{F}$ passes the zero spectrum criterion.
\end{lemma}
\begin{proof}
Assume that $\mathcal{F}$ does not pass the zero spectrum criterion. Then there are indices $i_1, \dots, i_9$ such that (\ref{line:1}) and (\ref{line:2}) hold. By Lemma \ref{lem:interzero}, $i_j \neq 1$,  for all $j \in \{4,\dots,9 \}$. Thus $$\sum_{k} N_{i_4,i_7}^{k} N_{i_5^*,i_8}^{k} N_{i_6,i_9^*}^{k} \ge N_{i_4,i_7}^{k_0} N_{i_5^*,i_8}^{k_0} N_{i_6,i_9^*}^{k_0} \neq 0$$
by assumption, which contradicts (\ref{line:2}).
\end{proof}

\begin{proposition} \label{prop:passzero}
All the fusion rings of the interpolated family ($\mathcal{R}_q$) pass the zero spectrum criterion.
\end{proposition}
\begin{proof}
For $q$ even or $q \equiv 1 \mod 4$, then we can directly apply Lemma \ref{lem:prezero}, with $x_{q+1,1}$ and $x_{q,1}$ respectively, from Theorems \ref{thm:q=0(2)} and \ref{thm:q=1(4)}. Now for $q \equiv -1 \mod 4$, we need to adapt a bit the argument. Consider Theorem \ref{thm:q=-1(4)}, and assume that the sum $\sum_{k} N_{i_4,i_7}^{k} N_{i_5^*,i_8}^{k} N_{i_6,i_9^*}^{k}$ does not involve $x_{\frac{q-1}{2},c_1}x_{\frac{q-1}{2},c_2}$ with $c_1=c_2$ (respectively with $c_1 \neq c_2$), then we can use the argument as for the proof of Lemma \ref{lem:prezero} with $x_{q+1,1}$ (respectively with $x_{q-1,c_3}$ and appropriate $c_3 \in \{1,\dots,\frac{q-3}{4}\}$, according to the second line in Theorem \ref{thm:q=-1(4)}, which is ok because we can take $q \ge 15$, so that $\frac{q-3}{4} \ge 3$). It remains the case where that sum involves $x_{\frac{q-1}{2},c_1}x_{\frac{q-1}{2},c_2}$ and $x_{\frac{q-1}{2},c_3}x_{\frac{q-1}{2},c_4}$, with $c_1=c_2$ and $c_3 \neq c_4$, where $c_i \in \{1,2\}$.  We will reach a contradiction. All the possible cases are similar, let write the details when $(i_4,i_7)$ and $(i_5^*,i_8)$ are given by $(x_{\frac{q-1}{2},c_1},x_{\frac{q-1}{2},c_2})$ and $(x_{\frac{q-1}{2},c_3},x_{\frac{q-1}{2},c_4})$ respectively: by (\ref{line:1}) and Frobenius reciprocity, $$N_{i_5,i_4}^{i_2} N_{i_8,i_7^*}^{i_2} = N_{i_5,i_4}^{i_2} N_{i_2,i_7}^{i_8} \neq 0,$$ 
so $i_2$ must be given by a simple component of $x_{\frac{q-1}{2},c_4}x_{\frac{q-1}{2},c_1}$ and $x_{\frac{q-1}{2},c_4}x_{\frac{q-1}{2},c_1}^*$, but there is no one by the first line of Theorem \ref{thm:q=-1(4)}, contradiction. \end{proof}

\subsection{One spectrum criterion}
Here is a general categorification criterion over every field. It corresponds to the existence of a pentagon equation of the form $0=xyz$, with $x,y,z \neq 0$ (see examples as $\mathcal{F}_{660}$ in \cite{LPR1}).

\begin{theorem}[see \cite{LPR1}] \label{thm:onecrit}
Let $\mathcal{F}$ be a fusion ring with fusion rules $(N_{i,j}^{k})$. If there are indices $i_0, i_1, \dots, i_9$ such that
\begin{align}
N_{i_4,i_1}^{i_6}, \ N_{i_5,i_4}^{i_2}, \ N_{i_5,i_6}^{i_3},& \ N_{i_7,i_9}^{i_1}, \ N_{i_2,i_7}^{i_8}, \ N_{i_8,i_9}^{i_3} \neq 0, \label{line:1'} \\
\sum_{k} N_{i_4,i_7}^{k} N_{i_5^*,i_8}^{k} N_{i_6,i_9^*}^{k}&=1, \label{line:2'} \\ 
N_{i_4,i_7}^{\spec}=N_{i_5^*,i_8}^{\spec}=N_{i_6,i_9^*}^{\spec}&=1, \label{line:3'}
\\
N_{i_2,i_1}^{i_{3}}&=0, \label{line:4'} \\
\sum_{k} N_{i_5,i_4}^{k} N_{i_8,i_7^*}^{k} &=1 
~\text{or}~
\sum_{k} N_{i_2,i_4^*}^{k} N_{i_8,\spec^*}^{k} =1 
~\text{or}~
\sum_{k} N_{i_5^*,i_2}^{k} N_{\spec,i_7^*}^{k} =1, 
\\
\sum_{k} N_{i_5,\spec}^{k} N_{i_3,i_9^*}^{k} &=1 
~\text{or}~
\sum_{k} N_{i_8,\spec^*}^{k} N_{i_3,i_6^*}^{k} =1 
~\text{or}~
\sum_{k} N_{i_5^*,i_8}^{k} N_{i_6,i_9^*}^{k} =1,
\\ 
\sum_{k} N_{i_4,i_7}^{k} N_{i_6,i_9^*}^{k} &=1 
~\text{or}~
\sum_{k} N_{\spec,i_7^*}^{k} N_{i_6,i_1^*}^{k} =1 
~\text{or}~
\sum_{k} N_{i_4^*,\spec}^{k} N_{i_1,i_9^*}^{k} =1. 
\end{align}
then $\mathcal{F}$ cannot be categorified at all, i.e. it is not the Grothendieck ring of a fusion category.
\end{theorem}

\begin{lemma} \label{lem:interone}
Let $\mathcal{F}$ be a fusion ring, and assume the existence of indices $i_1, \dots, i_9$ such that (\ref{line:1'}) and (\ref{line:4'}) hold. Then $i_j \neq 1$,  for all $j \in \{4,\dots,9 \}$.
\end{lemma}
\begin{proof}
Assume that there is $j \in \{4,\dots,9 \}$ such that $i_j=1$. We will reach a contradiction. Let write the details for $j=4$ (the other cases are similar): assume that $i_4=1$, by (\ref{line:1'}), $i_1=i_6$ and $i_2=i_5$. So, $N_{i_2,i_1}^{i_{3}} = N_{i_5,i_6}^{i_{3}} \neq 0$ by (\ref{line:1'}), but $N_{i_2,i_1}^{i_{3}} = 0$ by (\ref{line:4'}), contradiction.
\end{proof}

\begin{lemma} \label{lem:preone}
Let $\mathcal{F}$ be a fusion ring such that there are $k_0, k_1$ with $k_0 \neq k_1$, $N_{i,j}^{k_0}, N_{i,j}^{k_1} \neq 0$ for all $i,j \neq 1$. Then $\mathcal{F}$ passes the one spectrum criterion.
\end{lemma}
\begin{proof}
Assume that $\mathcal{F}$ does not pass the one spectrum criterion. Then there are indices $i_1, \dots, i_9$ such that (\ref{line:1'}), (\ref{line:2'}) and (\ref{line:4'}) hold. By Lemma \ref{lem:interone}, $i_j \neq 1$,  for all $j \in \{4,\dots,9 \}$. Thus $$\sum_{k} N_{i_4,i_7}^{k} N_{i_5^*,i_8}^{k} N_{i_6,i_9^*}^{k} \ge N_{i_4,i_7}^{k_0} N_{i_5^*,i_8}^{k_0} N_{i_6,i_9^*}^{k_0} + N_{i_4,i_7}^{k_1} N_{i_5^*,i_8}^{k_1} N_{i_6,i_9^*}^{k_1} \ge 2$$
by assumption, which contradicts (\ref{line:2'}).
\end{proof}

\begin{proposition} \label{prop:passone}
All the fusion rings of the interpolated family ($\mathcal{R}_q$) pass the one spectrum criterion.
\end{proposition}
\begin{proof}
For $q$ even: we can directly apply Lemma \ref{lem:preone}, with $x_{q+1,c}$ and $c \in \{1,\dots, \frac{q-2}{2}\}$ (from Theorem \ref{thm:q=0(2)}) because we can assume $q \ge 6$, so that $\frac{q-2}{2} \ge 2$. 

Next, for $q \equiv 1 \mod 4$: consider (\ref{line:2'}), then by Theorem \ref{thm:q=1(4)}, $N_{i_4,i_7}^{k} N_{i_5^*,i_8}^{k} N_{i_6,i_9^*}^{k} \neq 0$ if and only if $k = i_0$ given by $x_{q,1}$. Thus $ \sum_{k \neq i_0} N_{i_4,i_7}^{k} N_{i_5^*,i_8}^{k} N_{i_6,i_9^*}^{k}=0, $ and as for the argument of Proposition \ref{prop:passone}, it implies that the above sum must involve $x_{\frac{q-1}{2},c_1}x_{\frac{q-1}{2},c_2}$ and $x_{\frac{q-1}{2},c_3}x_{\frac{q-1}{2},c_4}$, with $c_1=c_2$ and $c_3 \neq c_4$, which also reach to a contradiction.

Finally, for $q \equiv -1 \mod 4$: if the sum (\ref{line:2'}) does not involve the product $x_{\frac{q-1}{2},c_1}x_{\frac{q-1}{2},c_2}$ then it must be at least $2$ by considering $x_{q,1}$ and $x_{q+1,1}$, ok. Else, if the sum (\ref{line:2'}) involves exactly one such product, and if $c_1 \neq c_2$ (respectively $c_1 = c_2$) then consider $x_{q+1,c_3}$ with $c_3 \in \{1, \dots , \frac{q-3}{4} \}$ (respectively $x_{q-1,c_3}$ with $c_3 \in \{1, \dots , \frac{q-3}{4} \}$ and $c_3 \neq \frac{q+1}{4}-c_2$), which is also ok because we can assume $q \ge 15$, so that $\frac{q-3}{4} \ge 3$. It remains that the sum (\ref{line:2'}) involves two such products, $x_{\frac{q-1}{2},c_1}x_{\frac{q-1}{2},c_2}$ and $x_{\frac{q-1}{2},c_3}x_{\frac{q-1}{2},c_4}$. If $c_1=c_2$ and $c_3=c_4$ (or both different) then applies previous step, else there is one equality and one non-equality, and by Theorem \ref{thm:q=1(4)}, the sum (\ref{line:2'}) must be zero.
\end{proof}

\section{Calculations} \label{sec:cal}
This section contains all the calculations required to prove Theorems \ref{thm:q=0(2)}, \ref{thm:q=-1(4)} and \ref{thm:q=1(4)}. The following lemmas compute every possible $\langle x_{d_1,c_1} x_{d_2,c_2}, x_{d_3,c_3} \rangle$, with inner product as in Corollary \ref{coro:inner}. We will make the calculation in full details for the first lemma (the other ones are similar). For the computation of the formal codegrees, see Remarks \ref{rem:orb-stab} and \ref{rk:interational}.
\subsection{$q$ even}
\begin{lemma} \label{lem:fullcomput}
$\langle x_{q-1,c_1} x_{q-1,c_2}, x_{q-1,c_3} \rangle = \left\{
    \begin{array}{lll}
        0 & \mbox{if} & c_1+c_2+c_3 = q+1 \mbox{ or } 2\Max(c_1,c_2,c_3), \\
        1 & \mbox{else}. &
    \end{array}
\right.$
\end{lemma}
\begin{proof} By Table \ref{table:q=0(2)}:
\begin{align*}
\langle x_{q-1,c_1} x_{q-1,c_2}, x_{q-1,c_3} \rangle 
& = \frac{(q-1)^3}{q(q^2-1)} - \frac{1}{q} - \frac{1}{q+1} \sum_{k=1}^{q/2}(\zeta_{q+1}^{kc_1} + \zeta_{q+1}^{-kc_1}) (\zeta_{q+1}^{kc_2} + \zeta_{q+1}^{-kc_2})(\zeta_{q+1}^{kc_3} + \zeta_{q+1}^{-kc_3}) \\
& = \frac{(q-1)^2}{q(q+1)} - \frac{1}{q} - \frac{1}{q+1} \sum_{k=1}^{q/2} \sum_{\epsilon_i=\pm 1} \zeta_{q+1}^{k\sum_{i=1}^3 \epsilon_i c_i} \\
& = \frac{(q-1)^2}{q(q+1)} - \frac{1}{q} - \frac{1}{q+1} \frac{1}{2} \sum_{\epsilon_i=\pm 1} \sum_{k=1}^{q} \zeta_{q+1}^{k\sum_{i=1}^3 \epsilon_i c_i}
\end{align*}
where $\sum_{\epsilon_i=\pm 1}$ means $\sum_{(\epsilon_1,\epsilon_2, \epsilon_3) \in \{-1, 1\}^3}$. The last equality comes from the fact that $\zeta_{q+1}^{ka} = \zeta_{q+1}^{(-k)(-a)} = \zeta_{q+1}^{(q+1-k)(-a)}$, whereas $(q+1-k)_{k=\frac{q}{2}, \dots, 1} = (k)_{k=\frac{q}{2} + 1, \dots, q} $.

For a fixed $(\epsilon_1,\epsilon_2, \epsilon_3) \in \{-1, 1\}^3$, consider $a := \sum_{i=1}^3 \epsilon_i c_i$, $d := \gcd(q+1, a)$, $d':=(q+1)/d$ and $b:=a/d$.
\begin{itemize}
\item If $d<q+1$, then $$\sum_{k=1}^{q}  \zeta_{q+1}^{ka} = \left( \sum_{k=1}^{q+1}  \zeta_{q+1}^{ka} \right) - 1 = \left( \sum_{r=0}^{d-1} \hspace*{.1cm} \sum_{k=rd'+1}^{(r+1)d'} \hspace*{-.1cm} \zeta_{d'}^{kb} \right) - 1 = \left( \sum_{r=1}^{d} \sum_{k=1}^{d'} \zeta_{d'}^{kb} \right) - 1 = \left( \sum_{r=1}^{d} 0 \right) - 1 = -1,$$
\item if $d=q+1$, then $q+1$ divides $a$, and so 
$$\sum_{k=1}^{q}  \zeta_{q+1}^{ka} = \sum_{k=1}^{q}  \zeta_{1}^{kb} =  \sum_{k=1}^{q} 1 = q.$$
\end{itemize}
Now, $1 \le c_i \le q/2$, so $q+1$ divides $a=\sum_{i=1}^3 \epsilon_i c_i$ if and only if $|a| \in \{ 0, q+1  \}$. But
\begin{itemize}
\item $|a| = q+1$ if and only if $c_1+c_2+c_3 = q+1$ and $(\epsilon_1,\epsilon_2, \epsilon_3) \in \{(-1,-1,-1),(1,1,1)\}$,
\item $a = 0$ if and only if $c_1+c_2+c_3 = 2\Max(c_1,c_2,c_3)$ and $(\epsilon_1,\epsilon_2, \epsilon_3)$ in some set of cardinal $2$ (again).
\end{itemize}
So
\begin{itemize}
\item if $\sum_{i=1}^3 c_i = q+1$ or $2\Max(c_1,c_2,c_3)$ then exactly $2$ elements $(\epsilon_1,\epsilon_2, \epsilon_3) \in \{-1, 1\}^3$ make $q+1$ divides $a$, and
\begin{align*}
\langle x_{q-1,c_1} x_{q-1,c_2}, x_{q-1,c_3} \rangle 
& =  \frac{(q-1)^2}{q(q+1)} - \frac{1}{q} - \frac{1}{q+1} \frac{1}{2} (2q-6)\\
& =  \frac{(q^2-2q+1) - (q+1) - q(q-3)}{q(q+1)} = 0.
\end{align*}
\item else
\begin{align*}
\langle x_{q-1,c_1} x_{q-1,c_2}, x_{q-1,c_3} \rangle 
& =  \frac{(q-1)^2}{q(q+1)} - \frac{1}{q} - \frac{1}{q+1} \frac{1}{2} (-8)\\
& =  \frac{(q^2-2q+1) - (q+1) + 4q}{q(q+1)} = \frac{q^2+q}{q(q+1)} = 1.     \qedhere
\end{align*}
\end{itemize}
\end{proof}

\begin{lemma}
$\langle x_{q-1,c_1} x_{q-1,c_2}, x_{q,1} \rangle = 1 - \delta_{c_1,c_2}$.
\end{lemma}
\begin{proof} As for the proof of Lemma \ref{lem:fullcomput}:
\begin{align*}
\langle x_{q-1,c_1} x_{q-1,c_2}, x_{q,1} \rangle
& = \frac{q(q-1)^2}{q(q^2-1)} - \frac{1}{q+1} \sum_{k=1}^{q/2} (\zeta_{q+1}^{k c_1} + \zeta_{q+1}^{-k c_1})(\zeta_{q+1}^{k c_2} + \zeta_{q+1}^{-k c_2}) \\
&  = \frac{q-1}{q+1} - \frac{1}{q+1} \frac{1}{2} (2q-2) = 0 \hspace{1cm}  \mbox{ if } c_1 = c_2, \\
&  =  \frac{q-1}{q+1} - \frac{1}{q+1} \frac{1}{2} (-4)  = 1 \hspace{2cm}  \mbox{ else. }  \qedhere
\end{align*}
\end{proof}

\begin{lemma}
$\langle x_{q-1,c_1} x_{q-1,c_2}, x_{q+1,c_3} \rangle = 1$.
\end{lemma}
\begin{proof} As for the proof of Lemma \ref{lem:fullcomput}:
$$\langle x_{q-1,c_1} x_{q-1,c_2}, x_{q+1,c_3} \rangle =  \frac{(q-1)^2(q+1)}{q(q^2-1)} + \frac{1}{q} = \frac{q-1}{q} + \frac{1}{q}  = 1.      \qedhere $$
\end{proof}

\begin{lemma}
$\langle x_{q,1} x_{q,1}, x_{q-1,c} \rangle = 1$.
\end{lemma}
\begin{proof} As for the proof of Lemma \ref{lem:fullcomput}:
$$\langle x_{q,1} x_{q,1}, x_{q-1,c} \rangle =  \frac{q^2(q-1)}{q(q^2-1)} - \frac{1}{q+1} \sum_{k=1}^{q/2}(\zeta_{q+1}^{kc} + \zeta_{q+1}^{-kc}) = \frac{q}{q+1} - \frac{1}{q+1} \frac{1}{2} (-2) = 1.    \qedhere $$
\end{proof}

\begin{lemma}
$\langle x_{q,1} x_{q,1}, x_{q,1} \rangle = 1$.
\end{lemma}
\begin{proof} As for the proof of Lemma \ref{lem:fullcomput}:
$$\langle x_{q,1} x_{q,1}, x_{q,1} \rangle =  \frac{q^3}{q(q^2-1)} + \frac{(q-2)/2}{q-1} - \frac{q/2}{q+1} = \frac{q^2 + (q-2)(q+1)/2 - q(q-1)/2}{q^2-1} = 1.   \qedhere $$
\end{proof}

\begin{lemma}
$\langle x_{q,1} x_{q,1}, x_{q+1,c} \rangle = 1$.
\end{lemma}
\begin{proof} As for the proof of Lemma \ref{lem:fullcomput}:
$$\langle x_{q,1} x_{q,1}, x_{q,1} \rangle =  \frac{q^2(q+1)}{q(q^2-1)} + \frac{1}{q-1} \sum_{k=1}^{(q-2)/2}(\zeta_{q-1}^{kc} + \zeta_{q-1}^{-kc}) = \frac{q}{q-1} + \frac{1}{q-1} \frac{1}{2} (-2) = 1.   \qedhere $$
\end{proof}

\begin{lemma}
$\langle x_{q+1,c_1} x_{q+1,c_2}, x_{q-1,c_3} \rangle = 1$.
\end{lemma}
\begin{proof} As for the proof of Lemma \ref{lem:fullcomput}:
$$\langle x_{q+1,c_1} x_{q+1,c_2}, x_{q-1,c_3} \rangle = \frac{(q-1)(q+1)^2}{q(q^2-1)} - \frac{1}{q} = \frac{q+1}{q} - \frac{1}{q} = 1.   \qedhere $$
\end{proof}

\begin{lemma}
$\langle x_{q+1,c_1} x_{q+1,c_2}, x_{q,1} \rangle = 1$.
\end{lemma}
\begin{proof} As for the proof of Lemma \ref{lem:fullcomput}:
$$\langle x_{q+1,c_1} x_{q+1,c_2}, x_{q,1} \rangle =  \frac{q(q+1)^2}{q(q^2-1)} + \frac{1}{q-1} \sum_{k=1}^{(q-2)/2}(\zeta_{q-1}^{kc_1} + \zeta_{q-1}^{-kc_1}) (\zeta_{q-1}^{kc_2} + \zeta_{q-1}^{-kc_2}) = \frac{q+1}{q-1} + \frac{1}{q-1} \frac{1}{2} (-4) = 1.    \qedhere $$
\end{proof}

\begin{lemma}
$\langle x_{q+1,c_1} x_{q+1,c_2}, x_{q,1} \rangle = 1 + \delta_{c_1,c_2}$.
\end{lemma}
\begin{proof} As for the proof of Lemma \ref{lem:fullcomput}:
\begin{align*}
\langle x_{q+1,c_1} x_{q+1,c_2}, x_{q,1} \rangle 
& =  \frac{q(q+1)^2}{q(q^2-1)} + \frac{1}{q-1} \sum_{k=1}^{(q-2)/2}(\zeta_{q-1}^{kc_1} + \zeta_{q-1}^{-kc_1}) (\zeta_{q-1}^{kc_2} + \zeta_{q-1}^{-kc_2}) \\
&  = \frac{q+1}{q-1} + \frac{1}{q-1} \frac{1}{2} (2(q-2)-2) = 2 \hspace{.8cm}  \mbox{ if } c_1 = c_2, \\
&  =  \frac{q+1}{q-1} + \frac{1}{q-1} \frac{1}{2} (-4) = 1 \hspace{2.5cm}  \mbox{ else. }    \qedhere
\end{align*}
\end{proof}

\begin{lemma} 
$\langle x_{q+1,c_1} x_{q+1,c_2}, x_{q+1,c_3} \rangle  = \left\{
    \begin{array}{lll}
        2 & \mbox{if} & c_1+c_2+c_3 = q-1 \mbox{ or } 2\Max(c_1,c_2,c_3) \\
        1 & \mbox{else} &
    \end{array}
\right.$.
\end{lemma}
\begin{proof} As for the proof of Lemma \ref{lem:fullcomput}:
\begin{align*}
\langle x_{q+1,c_1} x_{q+1,c_2}, x_{q+1,c_3} \rangle 
& = \frac{(q+1)^3}{q(q^2-1)} + \frac{1}{q} + \frac{1}{q-1} \sum_{k=1}^{(q-2)/2} (\zeta_{q-1}^{kc_1} + \zeta_{q-1}^{-kc_1}) (\zeta_{q-1}^{kc_2} + \zeta_{q-1}^{-kc_2}) (\zeta_{q-1}^{kc_3} + \zeta_{q-1}^{-kc_3}) \\
& = \frac{(q+1)^2}{q(q-1)} + \frac{1}{q} + \frac{1}{q-1} \frac{1}{2} \sum_{\epsilon_i=\pm 1} \sum_{k=1}^{q-2}  (\zeta_{q-1}^{\sum_{i=1}^3 \epsilon_i c_i})^k 
\end{align*}

For fixed $(\epsilon_i)$, let $d = \gcd(q-1, \sum_{i=1}^3 \epsilon_i c_i)$.
\begin{itemize}
\item if $d<q-1$, then $\sum_{k=1}^{q-2}  (\zeta_{q-1}^{\sum_{i=1}^3 \epsilon_i c_i})^k = -1,$
\item if $d=q-1$, then $\sum_{k=1}^{q-2}  (\zeta_{q-1}^{\sum_{i=1}^3 \epsilon_i c_i})^k = q-2$
\end{itemize}

Now, $1 \le c_i \le (q-2)/2$, so the only way for $q-1$ to divide $\sum_{i=1}^3 \epsilon_i c_i$ is to have $c_1+c_2+c_3 = q-1$ or $2\Max(c_1,c_2,c_3)$, and then there are exactly two possible $(\epsilon_i)$ for the divisibility to occur. Then
\begin{itemize}
\item if $\sum_{i=1}^3 c_i = q+1$ or $2\Max(c_1,c_2,c_3)$ then
\begin{align*}
\langle x_{q+1,c_1} x_{q+1,c_2}, x_{q+1,c_3} \rangle 
& =  \frac{(q+1)^2}{q(q-1)} + \frac{1}{q} + \frac{1}{q-1} \frac{1}{2} (2(q-2)-6)\\
& =  \frac{(q^2+2q+1) + (q-1) + (q^2-5q)}{q(q-1)} = \frac{2q^2-2q}{q(q-1)} = 2 
\end{align*}
\item else
\begin{align*}
\langle x_{q+1,c_1} x_{q+1,c_2}, x_{q+1,c_3} \rangle 
& =  \frac{(q+1)^2}{q(q-1)} + \frac{1}{q} + \frac{1}{q-1} \frac{1}{2} (-8)\\
& =  \frac{(q^2+2q+1) + (q-1) -4q}{q(q-1)} = \frac{q^2-q}{q(q-1)} = 1.     \qedhere 
\end{align*}
\end{itemize}
\end{proof}

\begin{lemma} \label{lem:last}
$\langle x_{q-1,c_1} x_{q,1}, x_{q+1,c_2} \rangle = 1$.
\end{lemma}
\begin{proof} 
As for the proof of Lemma \ref{lem:fullcomput}: $\langle x_{q-1,c_1} x_{q,1}, x_{q+1,c_2} \rangle  = \frac{(q-1)q(q+1)}{q(q^2-1)} = 1.$
\end{proof}

Now we prove the orthogonality relations in the interpolated case:

\begin{proposition}[Orthogonality Relations] \label{prop:ortho} For all $d_i,c_i$, $\langle x_{d_1,c_1}, x_{d_2,c_2} \rangle = \delta_{(d_1,c_1),(d_2,c_2)}.$
\end{proposition}
\begin{proof} 
As for the proof of Lemma \ref{lem:fullcomput}:
\begin{itemize}
\item $\langle x_{1,1}, x_{1,1} \rangle = \frac{1}{q(q^2-1)} + \frac{1}{q} + \frac{q-2}{2(q-1)} + \frac{q}{2(q+1)} = 1$,
\item $\langle x_{1,1}, x_{q-1,c} \rangle = \frac{q-1}{q(q^2-1)} - \frac{1}{q} - \frac{1}{q+1} \sum_k (\zeta_{q+1}^{kc} + \zeta_{q+1}^{-kc}) = \frac{q-1}{q(q^2-1)} - \frac{1}{q} - \frac{1}{q+1} \frac{1}{2} (-2) = 0$,
\item $\langle x_{1,1}, x_{q,1} \rangle = \frac{q}{q(q^2-1)} + \frac{q-2}{2(q-1)} - \frac{q}{2(q+1)} = 0$,
\item $\langle x_{1,1}, x_{q+1,c} \rangle = \frac{q+1}{q(q^2-1)} + \frac{1}{q} + \frac{1}{q-1} \sum_k (\zeta_{q-1}^{kc} + \zeta_{q-1}^{-kc}) =  \frac{q+1}{q(q^2-1)} + \frac{1}{q} + \frac{1}{q-1} \frac{1}{2} (-2) = 0$,
\item  $\langle x_{q-1,c_1}, x_{q-1,c_2} \rangle = \frac{(q-1)^2}{q(q^2-1)} + \frac{1}{q} + \frac{1}{q+1} \frac{1}{2}(2q \delta_{c_1,c_2} - (2+2(1-\delta_{c_1,c_2})) = \dots = \delta_{c_1,c_2}$,
\item $\langle x_{q-1,c_1}, x_{q,1} \rangle = \frac{(q-1)q}{q(q^2-1)} + \frac{1}{q+1} \frac{1}{2} (-2) = 0$,
\item $\langle x_{q-1,c_1}, x_{q+1,c_2} \rangle = \frac{(q-1)(q+1)}{q(q^2-1)} - \frac{1}{q} = 0$,
\item $\langle x_{q,1} , x_{q,1} \rangle =  \frac{q^2}{q(q^2-1)} + \frac{(q-2)/2}{q-1} + \frac{q/2}{q+1} = \cdots = 1,$
\item $\langle x_{q,1}, x_{q+1,c_1} \rangle = \frac{q(q+1)}{q(q^2-1)} + \frac{1}{q-1} \frac{1}{2} (-2) = 0$.
\item $\langle x_{q+1,c_1}, x_{q+1,c_2} \rangle = \frac{(q+1)^2}{q(q^2-1)} + \frac{1}{q} + \frac{1}{q-1} \frac{1}{2}(2(q-2) \delta_{c_1,c_2} - (2+2(1-\delta_{c_1,c_2})) = \dots = \delta_{c_1,c_2}$.   \qedhere
\end{itemize}
\end{proof}

Finally, let us compute the formal codegrees in the interpolate case:

\begin{proposition}[Formal codegrees] \label{prop:cod} The formal codegrees of the interpolated fusion rings $\mathcal{R}_q$
are the interpolation of the formal codegrees for $\Rep(\PSL(2,q))$. \end{proposition}
\begin{proof} The proof reduces to compute the squared norm of each column in Table \ref{table:q=0(2)}, and to observe that it is always $q(q^2 - 1)$ divided by the class size (see Remark \ref{rem:orb-stab}).
\begin{itemize}
\item $1 + \frac{q}{2} (q-1)^2 + q^2 + \frac{q-2}{2} (q+1)^2 = \cdots = q^3 - q = q(q^2 - 1)$,
\item $1 + \frac{q}{2} (-1)^2 + 0 + \frac{q-2}{2} (1)^2 = \cdots = q$,
\item $1 + 0 + 1 + \sum_{c=1}^{\frac{q-2}{2}} (\zeta_{q-1}^{kc} + \zeta_{q-1}^{-kc})^2 = 2 + \frac{1}{2}(2(q-2)-2) = q-1$, 
\item $1 + \sum_{c=1}^{\frac{q}{2}} (-\zeta_{q+1}^{kc} -\zeta_{q+1}^{-kc})^2 + (-1)^2 = 1 + \frac{1}{2}(2q-2) + 1 = q+1.$ \qedhere
\end{itemize}
\end{proof}

\begin{remark} \label{rk:interational}
In fact, the proof of Propositions \ref{prop:ortho} and \ref{prop:cod} did not require any calculation. Observe that the parameters $k$ and $c$ play the same role in the table of \S\ref{sec:tab}, so that all the required intermediate calculations were done before in the subsection. It is then clear that we are reduced to compute rational functions in $q$. But the expected results are already known to be true for $q$ prime-power (group case), so these rational functions are as expected for $q$ prime-power, so they must be so for all $q$.
\end{remark}

\subsection{$q \equiv -1 \mod 4$}

\begin{lemma} \label{lem:first2}
$\langle x_{\frac{q-1}{2},c_1} x_{\frac{q-1}{2},c_2}, x_{\frac{q-1}{2},c_3} \rangle = \delta_{c_1,c_2}(1-\delta_{c_1,c_3})$.
\end{lemma}
\begin{proof}
\begin{align*} \langle x_{\frac{q-1}{2},c_1} x_{\frac{q-1}{2},c_2}, x_{\frac{q-1}{2},c_3} \rangle = &
\frac{2(\frac{q-1}{2})^3}{q(q^2-1)} + \frac{1}{q} \sum_{k=1}^2(\frac{-1+i(-1)^{k+c_1}\sqrt{q}}{2})(\frac{-1+i(-1)^{k+c_2}\sqrt{q}}{2})(\frac{-1-i(-1)^{k+c_3}\sqrt{q}}{2}) \\ & + \frac{2}{q+1} \sum_{k=1}^{\frac{q-3}{4}}(-1)^{3(k+1)} +  \frac{(-1)^{3(\frac{q+1}{4}+1)}}{q+1}
\end{align*}
\begin{align*}
= \frac{(q-1)^2}{4q(q+1)} + \frac{1}{8q}\sum_{k=1}^2 (-1 & +i\sqrt{q}(-1)^k((-1)^{c_1} + (-1)^{c_2} - (-1)^{c_3}) \\ & + q(-1)^{2k}((-1)^{c_1+c_2}-(-1)^{c_1+c_3}-(-1)^{c_2+c_3}) \\ & + iq\sqrt{q}(-1)^{3k}(-1)^{c_1+c_2+c_3}) \\  + \frac{1}{q+1} &
\end{align*}
$$ =  \frac{(q-1)^2}{4q(q+1)} - \frac{1}{4q} + \frac{1}{4}((-1)^{c_1+c_2}-(-1)^{c_1+c_3}-(-1)^{c_2+c_3}) + \frac{1}{q+1}$$
Observe that $$(-1)^{c_1+c_2}-(-1)^{c_1+c_3}-(-1)^{c_2+c_3} = \left\{
    \begin{array}{lll}
        -1 & \mbox{if}  & (c_1,c_2,c_3) \neq (1,1,2), (2,2,1) \\
         3 & \mbox{else}&
    \end{array}
\right.$$
It follows that $\langle x_{\frac{q-1}{2},c_1} x_{\frac{q-1}{2},c_2}, x_{\frac{q-1}{2},c_3} \rangle =\left\{
    \begin{array}{lll}
        0 & \mbox{if}  & (c_1,c_2,c_3) \neq (1,1,2), (2,2,1) \\
        1 & \mbox{else}&
    \end{array}
\right. = \delta_{c_1,c_2}(1-\delta_{c_1,c_3})$.
\end{proof}

\begin{lemma} \label{lem:fullcomput2} 
$\langle x_{\frac{q-1}{2},c_1} x_{\frac{q-1}{2},c_2}, x_{q-1,c_3} \rangle = \delta_{c_1,c_2} $.
\end{lemma}
\begin{proof}
\begin{align*} \langle x_{\frac{q-1}{2},c_1} x_{\frac{q-1}{2},c_2}, x_{q-1,c_3} \rangle = &
\frac{2(q-1)(\frac{q-1}{2})^2}{q(q^2-1)} + \frac{1}{q} \sum_{k=1}^2(\frac{-1+i(-1)^{k+c_1}\sqrt{q}}{2})(\frac{-1+i(-1)^{k+c_2}\sqrt{q}}{2})(-1) \\ & + \frac{2}{q+1} \sum_{k=1}^{\frac{q-3}{4}}(-1)^{2(k+1)}(-\zeta_{q+1}^{2kc_3} - \zeta_{q+1}^{-2kc_3}) + \frac{(-1)^{2(\frac{q+1}{4}+1)}(-2(-1)^{c_3})}{q+1}
\end{align*}
\begin{align*}
= \frac{(q-1)^2}{2q(q+1)} & + \frac{1}{4q}\sum_{k=1}^2 (-1 +i\sqrt{q}(-1)^k((-1)^{c_1} + (-1)^{c_2}) + q(-1)^{2k}(-1)^{c_1+c_2}) \\ &  -\frac{2}{q+1} \sum_{\epsilon = \pm 1} \sum_{k=1}^{\frac{q-3}{4}}(\zeta_{\frac{q+1}{2}}^{ c_3})^{\epsilon k}  - \frac{2(-1)^{c_3}}{q+1} 
\end{align*}
Let $S$ be $\{1,2, \dots, \frac{q-3}{4}\}$. Then $-S \cup S = \{1,2, \dots, \frac{q+1}{2}\} \setminus \{ \frac{q+1}{4}, \frac{q+1}{2} \} \mod  \frac{q+1}{2}$. So:
\begin{align*}
\langle x_{\frac{q-1}{2},c_1} x_{\frac{q-1}{2},c_2}, x_{q-1,c_3} \rangle & = \frac{(q-1)^2}{2q(q+1)} + \frac{1}{2q}(-1 + q(-1)^{c_1+c_2})  + \frac{2}{q+1} ((-1)^{c_3}+1)  - \frac{2(-1)^{c_3}}{q+1} \\
& =  \frac{q^2-2q+1-q-1+4q+ (-1)^{c_1+c_2}(q^2+q)}{2q(q+1)} =  \left\{
    \begin{array}{ll}
        0 & \mbox{if }  c_1+c_2 \mbox{ odd} \\
        1 & \mbox{else} 
    \end{array}
\right. = \delta_{c_1,c_2},
\end{align*}
because $c_1,c_2 \in \{ 1,2 \}$, so $c_1+c_2$ is odd if and only if $c_1 \neq c_2$.
\end{proof}

\begin{lemma} 
$\langle x_{\frac{q-1}{2},c_1} x_{\frac{q-1}{2},c_2}, x_{q,1} \rangle = 0$.
\end{lemma}
\begin{proof}
\begin{align*} \langle x_{\frac{q-1}{2},c_1} x_{\frac{q-1}{2},c_2}, x_{q,c_3} \rangle & = 
\frac{2q(\frac{q-1}{2})^2}{q(q^2-1)} + \frac{2}{q+1} \sum_{k=1}^{\frac{q-3}{4}}(-1)^{2(k+1)}(-1) + \frac{(-1)^{2(\frac{q+1}{4}+1)}(-1)}{q+1} \\ & =  \frac{q-1}{2(q+1)} - \frac{2}{q+1}\frac{q-3}{4} - \frac{1}{q+1} = 0.   \qedhere
\end{align*}
\end{proof}

\begin{lemma} 
$\langle x_{\frac{q-1}{2},c_1} x_{\frac{q-1}{2},c_2}, x_{q+1,c_3} \rangle =  1 - \delta_{c_1,c_2}$.
\end{lemma}
\begin{proof}
\begin{align*} \langle x_{\frac{q-1}{2},c_1} x_{\frac{q-1}{2},c_2}, x_{q-1,c_3} \rangle = &
\frac{2(q+1)(\frac{q-1}{2})^2}{q(q^2-1)} + \frac{1}{q} \sum_{k=1}^2(\frac{-1+i(-1)^{k+c_1}\sqrt{q}}{2})(\frac{-1+i(-1)^{k+c_2}\sqrt{q}}{2})(1) \\ & =   \frac{q-1}{2q} - \frac{1}{2q}(-1 + q(-1)^{c_1+c_2}) = 1-\delta_{c_1,c_2}.   \qedhere
\end{align*}
\end{proof}

\begin{lemma} 
$\langle x_{\frac{q-1}{2},c_1} x_{q-1,c_2}, x_{q-1,c_3} \rangle = 1-\delta_{c_2,\frac{q+1}{4}-c_3}$.
\end{lemma}
\begin{proof}
\begin{align*} \langle x_{\frac{q-1}{2},c_1} x_{q-1,c_2}, x_{q-1,c_3} \rangle = &
\frac{2(q-1)^2(\frac{q-1}{2})}{q(q^2-1)} + \frac{1}{q} \sum_{k=1}^2(\frac{-1+i(-1)^{k+c_1}\sqrt{q}}{2})(-1)(-1) \\ & + \frac{2}{q+1} \sum_{k=1}^{\frac{q-3}{4}}(-1)^{k+1}(-\zeta_{q+1}^{2kc_2} - \zeta_{q+1}^{-2kc_2})(-\zeta_{q+1}^{2kc_3} - \zeta_{q+1}^{-2kc_3}) \\ & + \frac{(-1)^{\frac{q+1}{4}+1}(-2(-1)^{c_2})(-2(-1)^{c_3})}{q+1} \\
= & \frac{(q-1)^2}{q(q+1)} - \frac{1}{q} - \frac{2}{q+1} \sum_{\epsilon_2, \epsilon_3= \pm 1} \sum_{k=1}^{\frac{q-3}{4}} (\zeta_{\frac{q+1}{2}}^{\epsilon_2 c_2 + \epsilon_3 c_3 + \frac{q+1}{4}})^{k}  \\ & + \frac{4}{q+1}(-1)^{\frac{q+1}{4}+1+c_2+c_3}
\end{align*}
Note that $\zeta_{\frac{q+1}{2}}^{\frac{q+1}{4}} = \zeta_{\frac{q+1}{2}}^{-\frac{q+1}{4}}$, so: 
\begin{align*} \langle x_{\frac{q-1}{2},c_1} x_{q-1,c_2}, x_{q-1,c_3} \rangle = & \frac{(q-1)^2}{q(q+1)} - \frac{1}{q} + \frac{4}{q+1}(-1)^{\frac{q+1}{4}+1+c_2+c_3} \\ & - \frac{1}{q+1} \sum_{\epsilon_2, \epsilon_3= \pm 1} \sum_{\epsilon= \pm 1} \sum_{k=1}^{\frac{q-3}{4}} (\zeta_{\frac{q+1}{2}}^{\epsilon_2 c_2 + \epsilon_3 c_3 + \frac{q+1}{4}})^{\epsilon k}
\end{align*}
Now if $c_2+c_3 \not \equiv \frac{q+1}{4} \mod \frac{q+1}{2}$:  
$$ \sum_{\epsilon_2, \epsilon_3= \pm 1} \sum_{\epsilon= \pm 1} \sum_{k=1}^{\frac{q-3}{4}} (\zeta_{\frac{q+1}{2}}^{\epsilon_2 c_2 + \epsilon_3 c_3 + \frac{q+1}{4}})^{\epsilon k} 
= \sum_{\epsilon_2, \epsilon_3= \pm 1} ((-1)^{\epsilon_2 c_2 + \epsilon_3 c_3 + \frac{q+1}{4}}-1) = 4(-1)^{c_2 + c_3 + \frac{q+1}{4}} - 4$$
Else: 
$$ \sum_{\epsilon_2, \epsilon_3= \pm 1} \sum_{\epsilon= \pm 1} \sum_{k=1}^{\frac{q-3}{4}} (\zeta_{\frac{q+1}{2}}^{\epsilon_2 c_2 + \epsilon_3 c_3 + \frac{q+1}{4}})^{\epsilon k} 
= q-3 + \sum_{\substack{\epsilon_2, \epsilon_3= \pm 1 \\ \epsilon_2 \neq \epsilon_3}} ((-1)^{\epsilon_2 c_2 + \epsilon_3 c_3 + \frac{q+1}{4}}-1) = q-3 +2(-1)^{c_2 + c_3 + \frac{q+1}{4}} - 2 $$
But if $c_2+c_3 \not \equiv \frac{q+1}{4} \mod \frac{q+1}{2}$ then $c_2+c_3+\frac{q+1}{4}$ is even. So: 
$$\langle x_{\frac{q-1}{2},c_1} x_{q-1,c_2}, x_{q-1,c_3} \rangle =  \left\{
    \begin{array}{lll}
        1 & \mbox{if} & c_2+c_3 \not \equiv \frac{q+1}{4} \mod \frac{q+1}{2} \\
        0 & \mbox{else} &
    \end{array}
\right. = 1-\delta_{c_2,\frac{q+1}{4}-c_3},$$
because $1 \le c_i \le \frac{q-3}{4}$, for $i=2,3$.
\end{proof}

\begin{lemma} 
$\langle x_{\frac{q-1}{2},c_1} x_{q-1,c_2}, x_{q,1} \rangle = 1$.
\end{lemma}
\begin{proof}
\begin{align*} \langle x_{\frac{q-1}{2},c_1} x_{q-1,c_2}, x_{q,1} \rangle = &
\frac{2(q-1)q(\frac{q-1}{2})}{q(q^2-1)} + \frac{2}{q+1} \sum_{k=1}^{\frac{q-3}{4}}(-1)^{k+1}(-\zeta_{q+1}^{2kc_2} - \zeta_{q+1}^{-2kc_2})(-1) + \frac{(-1)^{\frac{q+1}{4}+1}(-2(-1)^{c_2})(-1)}{q+1}
\\ = & \frac{q-1}{q+1} - \frac{2}{q+1} \sum_{k=1}^{\frac{q-3}{4}} \sum_{\epsilon= \pm 1}  (\zeta_{\frac{q+1}{2}}^{c_2 + \frac{q+1}{4}})^{\epsilon k} - \frac{2(-1)^{c_2+\frac{q+1}{4}}}{q+1} = \frac{q-1}{q+1} + \frac{2}{q+1} = 1.   \qedhere
\end{align*}
\end{proof}

\begin{lemma} 
$\langle x_{\frac{q-1}{2},c_1} x_{q-1,c_2}, x_{q+1,c_3} \rangle = 1$.
\end{lemma}
\begin{proof}
\begin{align*} \langle x_{\frac{q-1}{2},c_1} x_{q-1,c_2}, x_{q-1,c_3} \rangle = &
\frac{2(q-1)(q+1)(\frac{q-1}{2})}{q(q^2-1)} + \frac{1}{q} \sum_{k=1}^2(\frac{-1+i(-1)^{k+c_1}\sqrt{q}}{2})(-1)(1)  \\ = & \frac{q-1}{q} + \frac{1}{q} = 1.   \qedhere
\end{align*}
\end{proof}

\begin{lemma}
$\langle x_{\frac{q-1}{2},c_1} x_{q,1}, x_{q,1} \rangle = 1$.
\end{lemma}
\begin{proof}
\begin{align*} \langle x_{\frac{q-1}{2},c_1}  x_{q,1}, x_{q,1} \rangle & = 
\frac{2q^2(\frac{q-1}{2})}{q(q^2-1)} + \frac{2}{q+1} \sum_{k=1}^{\frac{q-3}{4}}(-1)^{(k+1)}(-1)^2 + \frac{(-1)^{(\frac{q+1}{4}+1)}(-1)^2}{q+1} \\ & = \frac{q}{q+1} + \frac{1}{q+1} = 1.   \qedhere
\end{align*}
\end{proof}

\begin{lemma}
$\langle x_{\frac{q-1}{2},c_1} x_{q,1}, x_{q+1,c_2} \rangle = 1$.
\end{lemma}
\begin{proof}
$$\langle x_{\frac{q-1}{2},c_1} x_{q,1}, x_{q+1,c_2} \rangle = 
\frac{2q(q+1)(\frac{q-1}{2})}{q(q^2-1)} = 1.  \qedhere $$
\end{proof}

\begin{lemma}
$\langle x_{\frac{q-1}{2},c_1} x_{q+1,c_2}, x_{q+1,c_3} \rangle = 1$.
\end{lemma}
\begin{proof}
$$ \langle x_{\frac{q-1}{2},c_1} x_{q+1,c_2}, x_{q+1,c_3} \rangle = 
\frac{2(q+1)^2(\frac{q-1}{2})}{q(q^2-1)} + \frac{1}{q} \sum_{k=1}^2(\frac{-1+i(-1)^{k+c_1}\sqrt{q}}{2})(1)^2 = \frac{q+1}{q} - \frac{1}{q} = 1.   \qedhere $$
\end{proof}

\begin{lemma}
$\langle x_{q-1,c_1} x_{q-1,c_2}, x_{q-1,c_3} \rangle = \left\{
    \begin{array}{lll}
        1 & \mbox{if} & c_1+c_2+c_3 = \frac{q+1}{2} \mbox{ or } 2\Max(c_1,c_2,c_3) \\
        2 & \mbox{else} &
    \end{array}
\right.$.
\end{lemma}
\begin{proof}
\begin{align*} \langle x_{q-1,c_1} x_{q-1,c_2}, x_{q-1,c_3} \rangle = &
\frac{2(q-1)^3}{q(q^2-1)} + \frac{1}{q} \sum_{k=1}^2(-1)^3 + \frac{2}{q+1} \sum_{k=1}^{\frac{q-3}{4}}\prod_{i=1}^3(-\zeta_{q+1}^{2kc_i} - \zeta_{q+1}^{-2kc_i}) + \frac{1}{q+1}\prod_{i=1}^3(-2(-1)^{c_i}) \\
= & \frac{2(q-1)^2}{q(q+1)} - \frac{2}{q} - \frac{2}{q+1} \sum_{\epsilon_i= \pm 1} \sum_{k=1}^{\frac{q-3}{4}} (\zeta_{\frac{q+1}{2}}^{\sum_{i=1}^3 \epsilon_i c_i})^{k} - \frac{8}{q+1}(-1)^{\sum_{i=1}^3 c_i}
\end{align*}
As for the proof of Lemmas \ref{lem:fullcomput} and \ref{lem:fullcomput2}, if $\sum_{i=1}^3 c_i \neq \frac{q+1}{2}$ and $2\Max(c_1,c_2,c_3)$, then:  
\begin{align*} \langle x_{q-1,c_1} x_{q-1,c_2}, x_{q-1,c_3} \rangle = & \frac{2(q-1)^2}{q(q+1)} - \frac{2}{q} + \frac{8}{q+1} (1+(-1)^{\sum_{i=1}^3 c_i}) - \frac{8}{q+1}(-1)^{\sum_{i=1}^3 c_i} = 2.
\end{align*}
Else $\sum_{i=1}^3 c_i = \frac{q+1}{2}$ or $2\Max(c_1,c_2,c_3)$, in particular $\sum_{i=1}^3 c_i $ is even, and:
\begin{align*} \langle x_{q-1,c_1} x_{q-1,c_2}, x_{q-1,c_3} \rangle = & \frac{2(q-1)^2}{q(q+1)} - \frac{2}{q} + \frac{1}{q+1} (6(1+(-1)^{\sum_{i=1}^3 c_i}) - (q-3) ) - \frac{8}{q+1}(-1)^{\sum_{i=1}^3 c_i} 
\\ = & \frac{2(q-1)^2}{q(q+1)} - \frac{2}{q} + \frac{4-(q-3)}{q+1} = 1.   \qedhere
\end{align*}
\end{proof}

\begin{lemma}
$\langle x_{q-1,c_1} x_{q-1,c_2}, x_{q,1} \rangle = 2-\delta_{c_1,c_2}$.
\end{lemma}
\begin{proof} If $c_1 \neq c_2$:
\begin{align*} \langle x_{q-1,c_1} x_{q-1,c_2}, x_{q-1,c_3} \rangle = &
\frac{2q(q-1)^2}{q(q^2-1)} + \frac{2}{q+1} \sum_{k=1}^{\frac{q-3}{4}}(-1)\prod_{i=1}^2(-\zeta_{q+1}^{2kc_i} - \zeta_{q+1}^{-2kc_i}) + \frac{1}{q+1}(-1)\prod_{i=1}^2(-2(-1)^{c_i})
\\ = & \frac{2(q-1)}{q+1} + \frac{4}{q+1} (1+(-1)^{c_1 + c_2}) - \frac{4}{q+1}(-1)^{c_1 + c_2} = \frac{2(q-1)}{q+1} + \frac{4}{q+1} = 2.
\end{align*}
Else, $c_1 = c_2$, and:
$$ \langle x_{q-1,c_1} x_{q-1,c_2}, x_{q-1,c_3} \rangle = \frac{2(q-1)}{q+1} + \frac{2}{q+1} (1+(-1)^{2c_1} - \frac{q-3}{2}) - \frac{4}{q+1}(-1)^{2c_1} = 1.   \qedhere $$
\end{proof}

\begin{lemma}
$\langle x_{q-1,c_1} x_{q-1,c_2}, x_{q+1,c_3} \rangle = 2$.
\end{lemma}
\begin{proof}
$$ \langle x_{q-1,c_1} x_{q-1,c_2}, x_{q-1,c_3} \rangle = 
\frac{2(q+1)(q-1)^2}{q(q^2-1)} + \frac{1}{q} \sum_{k=1}^2(-1)^2(1) = \frac{2(q-1)}{q} + \frac{2}{q} = 2.   \qedhere
$$
\end{proof}

\begin{lemma}
$\langle x_{q-1,c_1} x_{q,1}, x_{q,1} \rangle = 2$.
\end{lemma}
\begin{proof}
\begin{align*} \langle x_{q-1,c_1} x_{q,1}, x_{q,1} \rangle & = 
\frac{2q^2(q-1)}{q(q^2-1)} + \frac{2}{q+1} \sum_{k=1}^{\frac{q-3}{4}}(-\zeta_{q+1}^{2kc_1} - \zeta_{q+1}^{-2kc_1})(-1)^2 + \frac{-2(-1)^{c_1}(-1)^2}{q+1} \\ & = \frac{2q}{q+1} + \frac{2}{q+1}(1+(-1)^{c_1}) - \frac{2}{q+1}(-1)^{c_1}b  = \frac{2q}{q+1} + \frac{2}{q+1} = 2.   \qedhere
\end{align*}
\end{proof}

\begin{lemma}
$\langle x_{q-1,c_1} x_{q,1}, x_{q+1,c_2} \rangle = 2$.
\end{lemma}
\begin{proof}
$$ \langle x_{q-1,c_1} x_{q,1},  x_{q+1,c_2} \rangle = 
\frac{2q(q+1)(q-1)}{q(q^2-1)} = 2.   \qedhere
$$
\end{proof}

\begin{lemma}
$\langle x_{q-1,c_1} x_{q+1,c_2}, x_{q+1,c_3} \rangle = 2$.
\end{lemma}
\begin{proof}
$$ \langle x_{q-1,c_1} x_{q+1,c_2}, x_{q+1,c_3} \rangle = 
\frac{2(q+1)^2(q-1)}{q(q^2-1)} + \frac{1}{q} \sum_{k=1}^2(-1)(1)^2 = \frac{2(q+1)}{q} - \frac{2}{q} = 2.   \qedhere
$$
\end{proof}

\begin{lemma}
$\langle x_{q,1} x_{q,1}, x_{q,1} \rangle = 2$.
\end{lemma}
\begin{proof}
\begin{align*} \langle x_{q-1,c_1} x_{q,1}, x_{q,1} \rangle & = 
\frac{2q^3}{q(q^2-1)} + \frac{2}{q-1} \sum_{k=1}^{\frac{q-3}{4}}(1)^3 + \frac{2}{q+1} \sum_{k=1}^{\frac{q-3}{4}}(-1)^3 + \frac{(-1)^3}{q+1} \\ & = \frac{2q^2}{q^2-1} + \frac{2}{q-1} \frac{q-3}{4} - \frac{2}{q+1} \frac{q-3}{4} - \frac{1}{q+1} = 2.   \qedhere
\end{align*}
\end{proof}

\begin{lemma}
$\langle x_{q,1} x_{q,1}, x_{q+1,c_1} \rangle = 2$.
\end{lemma}
\begin{proof}
\begin{align*} \langle x_{q,1} x_{q,1}, x_{q+1,c_1} \rangle & = 
\frac{2(q+1)q^2}{q(q^2-1)} + \frac{2}{q-1} \sum_{k=1}^{\frac{q-3}{4}}(1)^2 (\zeta_{q-1}^{2kc_1} + \zeta_{q-1}^{-2kc_1}).
\end{align*}
Let $S$ be $\{1,2, \dots, \frac{q-3}{4}\}$. Then $-S \cup S = \{1,2, \dots, \frac{q-1}{2}\} \setminus \{ \frac{q-1}{2} \} \mod  \frac{q-1}{2}$. So (as for Lemma \ref{lem:fullcomput2}):
$$ \langle x_{q,1} x_{q,1}, x_{q+1,c_1} \rangle  = 
\frac{2q}{q-1} + \frac{2}{q-1}(-1) = 2.   \qedhere
$$
\end{proof}

\begin{lemma}
$\langle x_{q,1} x_{q+1,c_1}, x_{q+1,c_2} \rangle = 2+\delta_{c_1,c_2}$.
\end{lemma}
\begin{proof} If $c_1 \neq c_2$:
\begin{align*} \langle x_{q,1} x_{q+1,c_1}, x_{q+1,c_2} \rangle & = 
\frac{2(q+1)^2q}{q(q^2-1)} + \frac{2}{q-1} \sum_{k=1}^{\frac{q-3}{4}}(1)(\zeta_{q-1}^{2kc_1} + \zeta_{q-1}^{-2kc_1})(\zeta_{q-1}^{2kc_2} + \zeta_{q-1}^{-2kc_2}) \\ & = \frac{2(q+1)}{q-1} + \frac{4}{q-1}(-1) = 2.
\end{align*}
Else $c_1 = c_2$, and:
$$ \langle x_{q,1} x_{q+1,c_1}, x_{q+1,c_2} \rangle = \frac{2(q+1)}{q-1} + \frac{2}{q-1}(-1 + \frac{q-3}{2}) = 3.   \qedhere
$$
\end{proof}

\begin{lemma} \label{lem:last2}
$\langle x_{q+1,c_1} x_{q+1,c_2}, x_{q+1,c_3} \rangle = \left\{
    \begin{array}{lll}
        3 & \mbox{if} & c_1+c_2+c_3 = \frac{q-1}{2} \mbox{ or } 2\Max(c_1,c_2,c_3) \\
        2 & \mbox{else} &
    \end{array}
\right.$
\end{lemma}
\begin{proof}
\begin{align*} \langle x_{q+1,c_1} x_{q+1,c_2}, x_{q+1,c_3} \rangle & = 
\frac{2(q+1)^3}{q(q^2-1)} + \frac{1}{q} \sum_{k=1}^2 (1)^3 + \frac{2}{q-1} \sum_{k=1}^{\frac{q-3}{4}}\prod_{i=1}^3 (\zeta_{q-1}^{2kc_i} + \zeta_{q-1}^{-2kc_i}) \\
& = \frac{2(q+1)^2}{q(q-1)} + \frac{2}{q} + \frac{2}{q-1} \sum_{\epsilon_i= \pm 1} \sum_{k=1}^{\frac{q-3}{4}} (\zeta_{\frac{q-1}{2}}^{\sum_{i=1}^3 \epsilon_i c_i})^{k}
\end{align*}
If $\sum_{i=1}^3 c_i \neq \frac{q-1}{2}$ and $2\Max(c_1,c_2,c_3)$, then (as for some previous lemmas): 
\begin{align*} \langle x_{q+1,c_1} x_{q+1,c_2}, x_{q+1,c_3} \rangle & = 
\frac{2(q+1)^2}{q(q-1)} + \frac{2}{q} + \frac{8}{q-1}(-1) = 2.
\end{align*}  
Else $\sum_{i=1}^3 c_i = \frac{q-1}{2}$ or $2\Max(c_1,c_2,c_3)$, and:
$$ \langle x_{q+1,c_1} x_{q+1,c_2}, x_{q+1,c_3} \rangle = 
\frac{2(q+1)^2}{q(q-1)} + \frac{2}{q} + \frac{1}{q-1}(q-3-6) = 3.   \qedhere
$$ 
\end{proof}

\begin{proposition}[Orthogonality Relations] \label{prop:ortho2} For all $d_i,c_i$, $\langle x_{d_1,c_1}, x_{d_2,c_2} \rangle = \delta_{(d_1,c_1),(d_2,c_2)}.$
\end{proposition}
\begin{proof}
As for Proposition \ref{prop:ortho}. Alternatively, see Remark \ref{rk:interational}.
\end{proof}

\begin{proposition}[Formal codegrees] \label{prop:cod2} The formal codegrees of the interpolated fusion rings
are the interpolation of the formal codegrees for $\Rep(\PSL(2,q))$. \end{proposition}
\begin{proof}
As for Proposition \ref{prop:cod}. Alternatively, see Remark \ref{rk:interational}.
\end{proof}

\subsection{$q \equiv 1 \mod 4$}
\begin{lemma} \label{lem:first3}
$\langle x_{\frac{q+1}{2},c_1} x_{\frac{q+1}{2},c_2}, x_{\frac{q+1}{2},c_3} \rangle = \delta_{c_1,c_2}\delta_{c_1,c_3}$.
\end{lemma}
\begin{proof}
\begin{align*} \langle x_{\frac{q+1}{2},c_1} x_{\frac{q+1}{2},c_2}, x_{\frac{q+1}{2},c_3} \rangle = &
\frac{2(\frac{q+1}{2})^3}{q(q^2-1)} + \frac{1}{q} \sum_{k=1}^2(\frac{1+(-1)^{k+c_1}\sqrt{q}}{2})(\frac{1+(-1)^{k+c_2}\sqrt{q}}{2})(\frac{1+(-1)^{k+c_3}\sqrt{q}}{2}) \\ & + \frac{2}{q-1} \sum_{k=1}^{\frac{q-5}{4}}(-1)^{3k} +  \frac{(-1)^{3\frac{q-1}{4}}}{q-1}
\end{align*}
\begin{align*}
= \frac{(q+1)^2}{4q(q-1)} + \frac{1}{8q}\sum_{k=1}^2 (1 & + \sqrt{q}(-1)^k((-1)^{c_1} + (-1)^{c_2} + (-1)^{c_3}) \\ & + q(-1)^{2k}((-1)^{c_1+c_2}+(-1)^{c_1+c_3}+(-1)^{c_2+c_3}) \\ & + q\sqrt{q}(-1)^{3k}(-1)^{c_1+c_2+c_3}) \\  - \frac{1}{q-1} &
\end{align*}
$$ =  \frac{(q+1)^2}{4q(q-1)} + \frac{1}{4q} + \frac{1}{4}((-1)^{c_1+c_2}+(-1)^{c_1+c_3}+(-1)^{c_2+c_3}) - \frac{1}{q-1}$$
Observe that $$(-1)^{c_1+c_2}+(-1)^{c_1+c_3}+(-1)^{c_2+c_3} = \left\{
    \begin{array}{lll}
        3 & \mbox{if}  & c_1=c_2=c_3 \\
        -1 & \mbox{else}&
    \end{array}
\right.$$
It follows that $\langle x_{\frac{q+1}{2},c_1} x_{\frac{q+1}{2},c_2}, x_{\frac{q+1}{2},c_3} \rangle =\left\{
    \begin{array}{lll}
        1 & \mbox{if}  & c_1=c_2=c_3 \\
        0 & \mbox{else}&
    \end{array}
\right. = \delta_{c_1,c_2}\delta_{c_1,c_3}$.
\end{proof}

\begin{lemma} 
$\langle x_{\frac{q+1}{2},c_1} x_{\frac{q+1}{2},c_2}, x_{q-1,c_3} \rangle = 1-\delta_{c_1,c_2}$.
\end{lemma}
\begin{proof}
\begin{align*} \langle x_{\frac{q+1}{2},c_1} x_{\frac{q+1}{2},c_2}, x_{q-1,c_3} \rangle = &
\frac{2(q-1)(\frac{q+1}{2})^2}{q(q^2-1)} + \frac{1}{q} \sum_{k=1}^2(\frac{1+(-1)^{k+c_1}\sqrt{q}}{2})(\frac{1+(-1)^{k+c_2}\sqrt{q}}{2})(-1) \\
= & \frac{q+1}{2q} - \frac{1}{4q} \sum_{k=1}^2 (1 + \sqrt{q} (-1)^k ((-1)^{c_1} + (-1)^{c_2}) + q(-1)^{2k}(-1)^{c_1+c_2}) \\
= &  \frac{q+1}{2q} - \frac{1+q(-1)^{c_1+c_2}}{2q} = 1-\delta_{c_1,c_2}
\end{align*}
because $c_1,c_2 \in \{1,2\}$, so that $c_1+c_2$ is even if and only if $c_1=c_2$.
\end{proof}

\begin{lemma} 
$\langle x_{\frac{q+1}{2},c_1} x_{\frac{q+1}{2},c_2}, x_{q,1} \rangle = 1$.
\end{lemma}
\begin{proof}
\begin{align*} \langle x_{\frac{q+1}{2},c_1} x_{\frac{q+1}{2},c_2}, x_{q,1} \rangle = &
\frac{2q(\frac{q+1}{2})^2}{q(q^2-1)} + \frac{2}{q-1} \sum_{k=1}^{\frac{q-5}{4}}(-1)^{2k} +  \frac{(-1)^{2\frac{q-1}{4}}}{q-1} \\
= & \frac{q+1}{2(q-1)} + \frac{2}{q-1}\frac{q-5}{4} + \frac{1}{q-1} = 1.   \qedhere
\end{align*}
\end{proof}

\begin{lemma} 
$\langle x_{\frac{q+1}{2},c_1} x_{\frac{q+1}{2},c_2}, x_{q+1,c_3} \rangle = \delta_{c_1,c_2}$.
\end{lemma}
\begin{proof}
\begin{align*} \langle x_{\frac{q+1}{2},c_1} x_{\frac{q+1}{2},c_2}, x_{q+1,c_3} \rangle = &
\frac{2(q+1)(\frac{q+1}{2})^2}{q(q^2-1)} + \frac{1}{q} \sum_{k=1}^2(\frac{1+(-1)^{k+c_1}\sqrt{q}}{2})(\frac{1+(-1)^{k+c_2}\sqrt{q}}{2}) \\ & + \frac{2}{q-1} \sum_{k=1}^{\frac{q-5}{4}}(-1)^{2k}(\zeta_{q-1}^{2kc_3} + \zeta_{q-1}^{-2kc_3}) +  \frac{(-1)^{2\frac{q-1}{4}}2(-1)^{c_3}}{q-1} \\
= &  \frac{(q+1)^2}{2q(q-1)} + \frac{1+q(-1)^{c_1+c_2}}{2q} - \frac{2}{q-1}(1+(-1)^{c_3}) + \frac{2}{q-1}(-1)^{c_3} \\
= & \frac{1}{2q(q-1)}(q^2+2q+1+q+q^2(-1)^{c_1+c_2}-1-q(-1)^{c_1+c_2}-4q) = \delta_{c_1,c_2}.   \qedhere
\end{align*}
\end{proof}

\begin{lemma}
$\langle x_{\frac{q+1}{2},c_1} x_{q-1,c_2}, x_{q-1,c_3} \rangle = 1$.
\end{lemma}
\begin{proof}
\begin{align*} \langle x_{\frac{q+1}{2},c_1} x_{q-1,c_2}, x_{q-1,c_3} \rangle = &
\frac{2(q-1)^2(\frac{q+1}{2})}{q(q^2-1)} + \frac{1}{q} \sum_{k=1}^2(\frac{1+(-1)^{k+c_1}\sqrt{q}}{2})(-1)(-1) \\
= & \frac{q-1}{q} + \frac{1}{2q} \sum_{k=1}^2 (1+(-1)^{k+c_1}\sqrt{q}) \\
= &  \frac{q-1}{q} + \frac{1}{q} = 1.   \qedhere
\end{align*}
\end{proof}

\begin{lemma}
$\langle x_{\frac{q+1}{2},c_1} x_{q-1,c_2}, x_{q,1} \rangle = 1$.
\end{lemma}
\begin{proof}
$$ \langle x_{\frac{q+1}{2},c_1} x_{q-1,c_2}, x_{q-1,c_3} \rangle = 
\frac{2q(q-1)(\frac{q+1}{2})}{q(q^2-1)} = 1.   \qedhere
$$
\end{proof}

\begin{lemma}
$\langle x_{\frac{q+1}{2},c_1} x_{q-1,c_2}, x_{q+1,c_3} \rangle = 1$.
\end{lemma}
\begin{proof}
\begin{align*} \langle x_{\frac{q+1}{2},c_1} x_{q-1,c_2}, x_{q+1,c_3} \rangle = &
\frac{2(q-1)(q+1)(\frac{q+1}{2})}{q(q^2-1)} + \frac{1}{q} \sum_{k=1}^2(\frac{1+(-1)^{k+c_1}\sqrt{q}}{2})(-1)(1) \\ 
= & \frac{q+1}{q} - \frac{1}{q} = 1.   \qedhere
\end{align*}
\end{proof}

\begin{lemma}
$\langle x_{\frac{q+1}{2},c_1} x_{q,1}, x_{q,1} \rangle = 1$.
\end{lemma}
\begin{proof}
$$ \langle x_{\frac{q+1}{2},c_1} x_{q,1}, x_{q,1} \rangle =
\frac{2q^2(\frac{q+1}{2})}{q(q^2-1)} + \frac{2}{q-1} \sum_{k=1}^{\frac{q-5}{4}}(-1)^{k}  + \frac{(-1)^{\frac{q-1}{4}}}{q-1} = \frac{q}{q-1} - \frac{1}{q-1} = 1.   \qedhere
$$
\end{proof}

\begin{lemma}
$\langle x_{\frac{q+1}{2},c_1} x_{q,1}, x_{q+1,c_2} \rangle = 1$.
\end{lemma}
\begin{proof}
\begin{align*} \langle x_{\frac{q+1}{2},c_1} x_{q,1}, x_{q+1,c_2} \rangle = &
\frac{2q(q+1)(\frac{q+1}{2})}{q(q^2-1)} + \frac{2}{q-1} \sum_{k=1}^{\frac{q-5}{4}}(-1)^{k}(\zeta_{q-1}^{2kc_2} + \zeta_{q-1}^{-2kc_2})  + \frac{(-1)^{\frac{q-1}{4}}2(-1)^{c_2}}{q-1} \\
= &  \frac{q+1}{q-1} - \frac{2}{q-1}(1+(-1)^{\frac{q-1}{4}+c_2}) + \frac{2(-1)^{\frac{q-1}{4}+c_2}}{q-1} = 1.   \qedhere
\end{align*}
\end{proof}

\begin{lemma}
$\langle x_{\frac{q+1}{2},c_1} x_{q+1,c_2}, x_{q+1,c_3} \rangle = 1+ \delta_{c_2+c_3,\frac{q-1}{4}}$.
\end{lemma}
\begin{proof}
\begin{align*} \langle x_{\frac{q+1}{2},c_1} x_{q+1,c_2}, x_{q+1,c_3} \rangle = &
\frac{2(q+1)^2(\frac{q+1}{2})}{q(q^2-1)} + \frac{1}{q} \sum_{k=1}^2(\frac{1+(-1)^{k+c_1}\sqrt{q}}{2}) + \frac{2}{q-1} \sum_{k=1}^{\frac{q-5}{4}}(-1)^{k}(\zeta_{q-1}^{2kc_2} + \zeta_{q-1}^{-2kc_2})(\zeta_{q-1}^{2kc_3} + \zeta_{q-1}^{-2kc_3}) \\ 
 &  + \frac{(-1)^{\frac{q-1}{4}}2(-1)^{c_2}2(-1)^{c_3}}{q-1} \\
= & \frac{(q+1)^2}{q(q-1)} + \frac{1}{q} + \frac{1}{q-1}(\delta_{c_2+c_3,\frac{q-1}{4}}(q-5-2(1+(-1)^{\frac{q-1}{4}+c_2+c_3})) \\ 
& - (1-\delta_{c_2+c_3,\frac{q-1}{4}})4(1+(-1)^{\frac{q-1}{4}+c_2+c_3})) + \frac{4(-1)^{\frac{q-1}{4}+c_2+c_3}}{q-1} \\
= & \frac{(q+1)^2}{q(q-1)} + \frac{1}{q} + \frac{\delta_{c_2+c_3,\frac{q-1}{4}}(q-5)-(1-\delta_{c_2+c_3,\frac{q-1}{4}})4}{q-1} \\
= & 2\delta_{c_2+c_3,\frac{q-1}{4}} + (1-\delta_{c_2+c_3,\frac{q-1}{4}}) = 1+ \delta_{c_2+c_3,\frac{q-1}{4}}.   \qedhere
\end{align*}
\end{proof}

\begin{lemma}
$\langle x_{q-1,c_1} x_{q-1,c_2}, x_{q-1,c_3} \rangle = \left\{
    \begin{array}{lll}
        1 & \mbox{if} & c_1+c_2+c_3 = \frac{q+1}{2} \mbox{ or } 2\Max(c_1,c_2,c_3) \\
        2 & \mbox{else} &
    \end{array}
\right.$.
\end{lemma}
\begin{proof}
\begin{align*} \langle x_{q-1,c_1} x_{q-1,c_2}, x_{q-1,c_3} \rangle = & 
\frac{2(q-1)^3}{q(q^2-1)} + \frac{1}{q} \sum_{k=1}^2(-1)^3 + \frac{2}{q+1} \sum_{k=1}^{\frac{q-1}{4}}\prod_{i=1}^3(-\zeta_{q+1}^{2kc_i} - \zeta_{q+1}^{-2kc_i}) \\
= & \frac{2(q-1)^2}{q(q+1)} - \frac{2}{q} - \frac{2}{q+1} \sum_{\epsilon_i= \pm 1} \sum_{k=1}^{\frac{q-1}{4}} (\zeta_{\frac{q+1}{2}}^{\sum_{i=1}^3 \epsilon_i c_i})^{k}
\end{align*}
If $\sum_{i=1}^3 c_i \neq \frac{q+1}{2}$ and $2\Max(c_1,c_2,c_3)$, then (as for some previous lemmas): 
\begin{align*} \langle x_{q-1,c_1} x_{q-1,c_2}, x_{q-1,c_3} \rangle = \frac{2(q-1)^2}{q(q+1)} - \frac{2}{q} - \frac{8}{q+1}(-1) = 2.
\end{align*}
Else $\sum_{i=1}^3 c_i = \frac{q+1}{2}$ or $2\Max(c_1,c_2,c_3)$, and then: 
$$ \langle x_{q-1,c_1} x_{q-1,c_2}, x_{q-1,c_3} \rangle = \frac{2(q-1)^2}{q(q+1)} - \frac{2}{q} - \frac{1}{q+1}(q-1-6) = 1.   \qedhere
$$
\end{proof}

\begin{lemma}
$\langle x_{q-1,c_1} x_{q-1,c_2}, x_{q,1} \rangle = 2-\delta_{c_1,c_2}$.
\end{lemma}
\begin{proof}
\begin{align*} \langle x_{q-1,c_1} x_{q-1,c_2}, x_{q,1} \rangle = & 
\frac{2q(q-1)^2}{q(q^2-1)} - \frac{2}{q+1} \sum_{k=1}^{\frac{q-1}{4}}\prod_{i=1}^2(-\zeta_{q+1}^{2kc_i} - \zeta_{q+1}^{-2kc_i}) \\
= & \frac{2(q-1)}{(q+1)} - \frac{1}{q+1}(\delta_{c_1,c_2}(q-1-2) + (1-\delta_{c_1,c_2})(-4)) = 2-\delta_{c_1,c_2}.   \qedhere
\end{align*}
\end{proof}

\begin{lemma}
$\langle x_{q-1,c_1} x_{q-1,c_2}, x_{q+1,c_3} \rangle = 2$.
\end{lemma}
\begin{proof}
$$ \langle x_{q-1,c_1} x_{q-1,c_2}, x_{q+1,c_3} \rangle =  
\frac{2(q+1)(q-1)^2}{q(q^2-1)} + \frac{1}{q} \sum_{k=1}^2(-1)^2 = \frac{2(q-1)}{q} + \frac{2}{q} = 2.   \qedhere
$$
\end{proof}

\begin{lemma}
$\langle x_{q-1,c_1} x_{q,1}, x_{q,1} \rangle = 2$.
\end{lemma}
\begin{proof}
\begin{align*} \langle x_{q-1,c_1} x_{q,1}, x_{q,1} \rangle = & 
\frac{2q^2(q-1)}{q(q^2-1)} + \frac{2}{q+1} \sum_{k=1}^{\frac{q-1}{4}}(-\zeta_{q+1}^{2kc_1} - \zeta_{q+1}^{-2kc_1}) \\
= & \frac{2q}{(q+1)} - \frac{2}{q+1}(-1) = 2.   \qedhere
\end{align*}
\end{proof}

\begin{lemma}
$\langle x_{q-1,c_1} x_{q,1}, x_{q+1,c_2} \rangle = 2$.
\end{lemma}
\begin{proof}
$$ \langle x_{q-1,c_1} x_{q,1}, x_{q+1,c_2} \rangle = 
\frac{2q(q+1)(q-1)}{q(q^2-1)} = 2.   \qedhere
$$
\end{proof}

\begin{lemma}
$\langle x_{q-1,c_1} x_{q+1,c_2}, x_{q+1,c_3} \rangle = 2$.
\end{lemma}
\begin{proof}
$$ \langle x_{q-1,c_1} x_{q+1,c_2}, x_{q+1,c_3} \rangle = 
\frac{2(q+1)^2(q-1)}{q(q^2-1)} + \frac{1}{q} \sum_{k=1}^{2} (-1) = \frac{2(q+1)}{q} - \frac{2}{q} = 2.   \qedhere
$$
\end{proof}

\begin{lemma}
$\langle x_{q,1} x_{q,1}, x_{q,1} \rangle = 2$.
\end{lemma}
\begin{proof}
\begin{align*} \langle x_{q,1} x_{q,1}, x_{q,1} \rangle = & 
\frac{2q^3}{q(q^2-1)} + \frac{2}{q-1}\sum_{k=1}^{\frac{q-5}{4}}(1)^3 + \frac{1}{q-1}(1)^3 + \frac{2}{q+1}\sum_{k=1}^{\frac{q-1}{4}}(-1)^3 \\
= & \frac{2q^2}{q^2-1} + \frac{2}{q-1}\frac{q-5}{4} + \frac{1}{q-1} - \frac{2}{q+1}\frac{q-1}{4} = 2.   \qedhere
\end{align*}
\end{proof}

\begin{lemma}
$\langle x_{q,1} x_{q,1}, x_{q+1,c_1} \rangle = 2$.
\end{lemma}
\begin{proof}
\begin{align*} \langle x_{q,1} x_{q,1}, x_{q+1,c_1} \rangle = & 
\frac{2(q+1)q^2}{q(q^2-1)} + \frac{2}{q-1}\sum_{k=1}^{\frac{q-5}{4}}(\zeta_{q-1}^{2kc_1} + \zeta_{q-1}^{-2kc_1}) + \frac{1}{q-1}2(-1)^{c_1} \\
= & \frac{2q}{q-1} - \frac{2}{q-1}(1+(-1)^{c_1}) + \frac{2(-1)^{c_1}}{q-1} = 2.   \qedhere
\end{align*}
\end{proof}

\begin{lemma}
$\langle x_{q,1} x_{q+1,c_1}, x_{q+1,c_2} \rangle = 2+\delta_{c_1,c_2}$.
\end{lemma}
\begin{proof}
\begin{align*} \langle x_{q,1} x_{q+1,c_1}, x_{q+1,c_2} \rangle = & 
\frac{2(q+1)^2q}{q(q^2-1)} + \frac{2}{q-1}\sum_{k=1}^{\frac{q-5}{4}}(\zeta_{q-1}^{2kc_1} + \zeta_{q-1}^{-2kc_1})(\zeta_{q-1}^{2kc_2} + \zeta_{q-1}^{-2kc_2}) + \frac{1}{q-1}2(-1)^{c_1}2(-1)^{c_2} \\
= & \frac{2(q+1)}{q-1} - \frac{2((1-\delta_{c_1,c_2})2(1+(-1)^{c_1+c_2}) + \delta_{c_1,c_2}(1+(-1)^{c_1+c_2}-\frac{q-5}{2}))}{q-1} + \frac{4(-1)^{c_1+c_2}}{q-1} \\
= & 2+\delta_{c_1,c_2}.   \qedhere
\end{align*}
\end{proof}

\begin{lemma}  \label{lem:last3}
$\langle x_{q+1,c_1} x_{q+1,c_2}, x_{q+1,c_3} \rangle = \left\{
    \begin{array}{lll}
        3 & \mbox{if} & c_1+c_2+c_3 = \frac{q-1}{2} \mbox{ or } 2\Max(c_1,c_2,c_3) \\
        2 & \mbox{else} &
    \end{array}
\right.$.
\end{lemma}
\begin{proof}
\begin{align*} \langle x_{q+1,c_1} x_{q+1,c_2}, x_{q+1,c_3} \rangle = & 
\frac{2(q+1)^3}{q(q^2-1)} + \frac{1}{q} \sum_{k=1}^2 (1)^3 + \frac{2}{q-1}\sum_{k=1}^{\frac{q-5}{4}} \prod_{i=1}^3(\zeta_{q-1}^{2kc_i} + \zeta_{q-1}^{-2kc_i}) + \frac{1}{q-1}8(-1)^{\sum_{i=1}^3 c_i}
\end{align*}
If $\sum_{i=1}^3 c_i \neq \frac{q-1}{2}$ and $2\Max(c_1,c_2,c_3)$, then (as for some previous lemmas): 
\begin{align*} \langle x_{q+1,c_1} x_{q+1,c_2}, x_{q+1,c_3} \rangle = \frac{2(q+1)^2}{q(q-1)} + \frac{2}{q} - \frac{8(1+(-1)^{\sum_{i=1}^3 c_i})}{q-1}  + \frac{8(-1)^{\sum_{i=1}^3 c_i}}{q-1} = 2\frac{q^2+2q+1+q-1-4q}{q(q-1)} = 2.
\end{align*}
Else $\sum_{i=1}^3 c_i = \frac{q-1}{2}$ or $2\Max(c_1,c_2,c_3)$, in particular $\sum_{i=1}^3 c_i$ is even,  and then: 
\begin{align*} \langle x_{q+1,c_1} x_{q+1,c_2}, x_{q+1,c_3} \rangle = & \frac{2(q+1)^2}{q(q-1)} + \frac{2}{q} + \frac{q-5-6(1+(-1)^{\sum_{i=1}^3 c_i})}{q-1}  + \frac{8(-1)^{\sum_{i=1}^3 c_i}}{q-1} \\
 = & \frac{2q^2+4q+2+2q-2+q^2-9q}{q(q-1)} = 3.   \qedhere
\end{align*}
\end{proof}

\begin{proposition}[Orthogonality Relations] \label{prop:ortho3} For all $d_i,c_i$, $\langle x_{d_1,c_1}, x_{d_2,c_2} \rangle = \delta_{(d_1,c_1),(d_2,c_2)}.$
\end{proposition}
\begin{proof} 
As for Proposition \ref{prop:ortho}. Alternatively, see Remark \ref{rk:interational}.
\end{proof}

\begin{proposition}[Formal codegrees] \label{prop:cod3} The formal codegrees of the interpolated fusion rings
are the interpolation of the formal codegrees for $\Rep(\PSL(2,q))$. \end{proposition}
\begin{proof}
As for Proposition \ref{prop:cod}. Alternatively, see Remark \ref{rk:interational}.
\end{proof}

\section{Appendix}

\subsection{Etingof's moderation} \label{EtiMod}
The previous sections shows that the Grothendieck rings of the family of fusion categories  $\Rep(\PSL(2,q))$, with $q$ prime-power, interpolates to a family of fusion rings $(\mathcal{R}_q)$ for every integer $q \ge 2$, which moreover pass all the known criteria. P. Etingof pointed out \cite{eti20} that such a result is not sufficient to believe on the existence of a non prime-power $q$ for which $\mathcal{R}_q$ admits a complex categorification; because, as we will see in this section, the family of fusion categories  $\Rep(\mathbb{F}_q \rtimes \mathbb{F}_q^{\star})$, with $q$ prime-power, interpolates to a family of fusion rings $(\mathcal{T}_q)$ for every integer $q \ge 2$, which also pass all the known criteria, but according to \cite[Corollary 7.4]{EGO04}, the ones which admit a complex categorification are exactly those with $q$ prime-power.

The Grothendieck ring of $\Rep(\mathbb{F}_q \rtimes \mathbb{F}_q^{\star})$ is of rank $q$, $\FPdim \ q(q-1)$, type $[[1,q-1],[q-1,1]]$ and fusion rules: 
$$\left\{
    \begin{array}{lll}
x_{1,c_1} x_{1,c_2} & = x_{1,c_1+c_2 \hspace{-.2cm} \mod q-1}, \\
x_{1,c} x_{q-1,1} & = x_{q-1,1} x_{1,c} = x_{q-1,1}, \\
x_{q-1,1} x_{q-1,1} & = \sum_c x_{1,c} + (q-2)x_{q-1,1}. \\
    \end{array}
\right.$$
Then, the generic character table of $\mathbb{F}_q \rtimes \mathbb{F}_q^{\star}$ is 
$$
\begin{tabular}{ |c||c|c|c|c|c| } 
\hline
\backslashbox{\text{charparam }$c$}{\text{classparam }$k$} & $\{1\}$ & $\{1, \dots, q-2\}$ & $\{1\}$  \\   \hline \hline
 $\{0, \dots, q-2\}$ & $1$ & $\zeta_{q-1}^{kc}$ & $1$ \\ 		 
 $\{1\}$ & $q-1$ & $0$ & $-1$  \\ \hline \hline	  
  \text{ class size } & $1$ & $q$ & $q-1$  \\    \hline
\end{tabular}
$$
The above family of Grothendieck rings with $q$ prime-power interpolates into a family of fusion rings $(\mathcal{T}_q)$ for all integer $q \ge 2$, with eigentable as above. Let us see why they all pass every categorification criterion mentioned in \S\ref{sec:criteria}:  

\begin{itemize}
\item Schur product criterion: if $(j_1,j_2,j_3)$ from Theorem \ref{thm:schur} involves the middle column of the eigentable then the presence of $0$ reduces the checking to the one for the cyclic group $C_{q-1}$, which is true. It remains only four cases, whose computation give $q(q-1), 0, q, \frac{q(q-2)}{q-1}$ respectively, and all of them are non-negative.
\item Ostrik criterion: $2\sum_j \frac{1}{\mathfrak{c}_j^2} = 2(\frac{1}{q^2(q-1)^2} + \frac{q-2}{(q-1)^2} + \frac{1}{q^2}) = \frac{2(q^3-q^2-2q+2)}{q^2(q-1)^2} \le 1 < 1 + \frac{1}{\mathfrak{c}_1}$.
\item Drinfeld center criterion: $(\mathfrak{c}_1/\mathfrak{c}_j)$ are the class sizes mentioned in the eigentable, they are all integers.
\item Isaacs criterion: observe that $\frac{\mathfrak{c}_1/\mathfrak{c}_j}{\lambda_{i,1}}$ is always an integer when $\lambda_{i,j} \neq 0$.
\item Extended cyclotomic criterion: all the entries of the eigentable are cyclotomic integers.
\item Zero spectrum criterion: see Lemma \ref{FFzero}.
\item One spectrum criterion: see Lemma \ref{FFone}.
\end{itemize}

\begin{lemma} \label{FFzero}
The fusion ring $\mathcal{T}_q$ passes the zero spectrum criterion for all $q \ge 2$.
\end{lemma}
\begin{proof}
Consider $i_1, \dots, i_9$ of Theorem \ref{thm:zerocrit}. If none of them is given by $x_{q-1,1}$ then it reduces to the cyclic group $C_{q-1}$. Else, one of them is given by $x_{q-1,1}$, and then it is easy to see by the fusion rules together with (\ref{line:1}), that at least two among $i_4, \dots , i_9$ are given by $x_{q-1,1}$: 
\begin{itemize}
\item if there are exactly two, and if they are in the (bottom) indices of the same coefficient in (\ref{line:2}), say $i_4$ and $i_7$, then $i_j$, with $j=5,6,8,9$, are given by some $x_{1,c_j}$, with $c_6-c_9 \not \equiv c_8-c_5 \mod q-1$, contradiction with $c_5 + c_6 \equiv c_8 + c_9$ given by $N_{i_5,i_6}^{i_3},N_{i_8,i_9}^{i_3} \neq 0 $ at (\ref{line:1}); else they are in the indices of different coefficients in (\ref{line:2}), say $i_4$ and $i_6$, and then $i_3$ must be given by $x_{q-1,1}$ because $N_{i_5,i_6}^{i_3} \neq 0$, contradiction with $N_{i_8,i_9}^{i_3} \neq 0$. 
\item if there are exactly three, and if they are in the indices of three different coefficients in (\ref{line:2}) then this sum cannot be zero, so they must be in the indices of exactly two different coefficients in (\ref{line:2}), say $i_4, i_7, i_6$, then we also get a contradiction using $N_{i_5,i_6}^{i_3},N_{i_8,i_9}^{i_3} \neq 0 $. 
\item if there are strictly more than three, then (as above), they must be in the indices of exactly two different coefficients in (\ref{line:2}), which also prevents the sum to be zero. \qedhere
\end{itemize}
\end{proof}

\begin{lemma} \label{FFone}
The fusion ring $\mathcal{T}_q$ passes the one spectrum criterion for all $q \ge 2$.
\end{lemma}
\begin{proof}
As for the proof of Lemma \ref{FFzero}, we can assume that at least two among $i_4, \dots , i_9$ are given by $x_{q-1,1}$. If $i_0$ is given by $x_{q-1,1}$ then the indices of each coefficient of (\ref{line:2'}) must involve at least one $x_{q-1,1}$: 
\begin{itemize}
\item if it is exactly one each, say $i_4, i_5, i_6$, then by (\ref{line:1'}) and using the notations of the proof of Lemma \ref{FFzero}, $c_1 \equiv c_7 + c_8 $,   $c_2 \equiv c_8 - c_7 $ and $c_3 \equiv c_8 + c_9 $, so that $c_2+c_1 \equiv c_3 \mod q-1$, contradiction with (\ref{line:4'}).
\item if there is one coefficient for two ones, and the two others with one, say $i_4, i_5, i_6, i_7$, then by $N_{i_7,i_9}^{i_1}, N_{i_2,i_7}^{i_8} \neq 0$ in (\ref{line:1'}), $i_1$ and $i_2$ must be given by $x_{q-1,1}$ too, contradiction with (\ref{line:4'}).
\item if there is two coefficients for two ones, and the other with one, say $i_4, i_5, i_6, i_7, i_8$, then by $N_{i_7,i_9}^{i_1}, \ N_{i_8,i_9}^{i_3} \neq 0$ in (\ref{line:1'}), $i_1$ and $i_3$ must be given by $x_{q-1,1}$ too, contradiction with (\ref{line:4'}).
\item else it is exactly two each, the sum in (\ref{line:2'}) cannot be equal to one.
\end{itemize}
Then $i_0$ is not given by $x_{q-1,1}$, and so the bottom indices of each coefficient in (\ref{line:2'}) must be both $x_{q-1,1}$ or both not: 
\begin{itemize}
\item if it is ``both not" for all, we reduce to the cyclic group $C_{q-1}$,
\item if it is ``both" for exactly one, say $i_4$ and $i_7$, then by $N_{i_7,i_9}^{i_1}, N_{i_2,i_7}^{i_8} \neq 0$ in (\ref{line:1'}), $i_1$ and $i_2$ must be given by $x_{q-1,1}$ too, contradiction with (\ref{line:4'}),
\item if it is ``both" for exactly two, say $i_4, i_7, i_5, i_8$, then by $N_{i_7,i_9}^{i_1}, \ N_{i_8,i_9}^{i_3} \neq 0$ in (\ref{line:1'}), $i_1$ and $i_3$ must be given by $x_{q-1,1}$ too, contradiction with (\ref{line:4'}),
\item else it is ``both" for all, then the sum in (\ref{line:2'}) cannot be equal to one. \qedhere
\end{itemize}
\end{proof}

\subsection{L\"ubeck's trick} 
Here is a trick (given by Frank L\"ubeck) to display the character table directly on GAP for $q$ not necessarily a prime-power, by removing the prime-power restriction as follows:  
\begin{verbatim}
gap> gt02 := CharacterTableFromLibrary("SL2even");;
gap> gt02.domain := function(q) return q mod 2 = 0; end;;
gap> gt1 := CharacterTableFromLibrary("PSL2even");;
gap> gt1.domain := function(q) return q mod 4 = 1; end;;
gap> gt3 := CharacterTableFromLibrary("PSL2odd");;
gap> gt3.domain := function(q) return q mod 4 = 3; end;;  
\end{verbatim} 
Then we can display as usual for every integer $q$ using the following command:
\begin{verbatim}
gap> Display(CharacterTable("PSL",2,q));
\end{verbatim} 

In addition, we can get the fusion ring for every \textit{given} integer $q$ by applying the Schur orthogonality relations with the following function: 
\begin{verbatim}
gap> FusionPSL2:=function(q)
> Ch:=CharacterTableWithSortedCharacters(CharacterTable("PSL",2,q));;
> irr:=Irr(Ch);; n:=Size(irr);;
> return List([1..n],i->List([1..n],j->List([1..n],k->ScalarProduct(irr[i]*irr[j],irr[k]))));
> end;;
\end{verbatim}
\begin{remark}
Theorems \ref{thm:q=0(2)}, \ref{thm:q=-1(4)} and \ref{thm:q=1(4)} provide the generic fusion rules, whereas above function compute them for every \textit{given} integer. We used it to double-check (for $q$ small) that these theorems contain no mistake.
\end{remark}

\end{document}